\documentclass[journal]{IEEEtran}
\usepackage{amsmath}
\usepackage{graphicx}
\usepackage{amssymb}
\usepackage{epstopdf}
\usepackage{amsthm}
\usepackage{bm}
\usepackage{algorithm}
\usepackage{enumitem}
\usepackage{cite}
\usepackage[noend]{algpseudocode}
\usepackage{soul,color}
\usepackage{mathtools}
\usepackage{caption}
\usepackage{subcaption}
\usepackage{cases}
\usepackage{tabularx}
\usepackage{makecell}

\def\P{{\cal P}}

\def\B{{\cal B}}
\def\R^n{{\cal U}}
\def\R^m{{\cal V}}

\def\R{{\mathbb R}}

\def\C{{\cal C}}

\newcommand{\unorm}[1]{\lVert#1\rVert}
\newcommand{\uabs}[1]{\lvert#1\rvert}
\newcommand{\norm}[1]{\left\lVert#1\right\rVert}
\newcommand{\abs}[1]{\left\lvert#1\right\rvert}

\newcommand{\topnew}{\top \hspace{-0.05cm}}

\newcommand\numberthis{\addtocounter{equation}{1}\tag{\theequation}}

\renewcommand{\arraystretch}{1.5}
\algnewcommand{\IfThenElse}[3]{
  \State \algorithmicif\ #1\ \algorithmicthen\ #2\ \algorithmicelse\ #3}

\DeclareMathOperator*{\argmin}{argmin}

\DeclareMathOperator*{\vect}{vec}
\DeclareMathOperator*{\rank}{rank}

\DeclareMathOperator*{\diag}{diag}

\DeclareMathOperator*{\sing}{singleton}

\newtheorem{theorem}{Theorem}
\newtheorem{proposition}{Proposition}
\newtheorem{lemma}{Lemma}
\newtheorem{example}{Example}
\newtheorem{remark}{Remark}
\newtheorem{corollary}{Corollary}
\newtheorem{definition}{Definition}

\ifCLASSINFOpdf
\else
\fi

\begin{document}
%
\title{On Asymptotic Linear Convergence of Projected Gradient Descent for Constrained Least Squares}
%
%
%

\author{Trung~Vu,~\IEEEmembership{Graduate~Student~Member,~IEEE},
    Raviv~Raich,~\IEEEmembership{Senior~Member,~IEEE}
\thanks{Trung Vu and Raviv Raich are with the School of Electrical Engineering and Computer Science, Oregon State University, Corvallis, OR 97331, USA. (e-mails: vutru@oregonstate.edu and raich@oregonstate.edu).}}

%
%

\markboth{Preprint}%
{Vu \MakeLowercase{\textit{et al.}}: On Local Linear Convergence of Projected Gradient Descent for Constrained Least Squares}
%



\maketitle

\begin{abstract}
Many recent problems in signal processing and machine learning such as compressed sensing, image restoration, matrix/tensor recovery, and non-negative matrix factorization can be cast as constrained optimization. Projected gradient descent is a simple yet efficient method for solving such constrained optimization problems. Local convergence analysis furthers our understanding of its asymptotic behavior near the solution, offering sharper bounds on the convergence rate compared to global convergence analysis. However, local guarantees often appear scattered in problem-specific areas of machine learning and signal processing. This manuscript presents a unified framework for the local convergence analysis of projected gradient descent in the context of constrained least squares. The proposed analysis offers insights into pivotal local convergence properties such as the conditions for linear convergence, the region of convergence, the exact asymptotic rate of convergence, and the bound on the number of iterations needed to reach a certain level of accuracy. To demonstrate the applicability of the proposed approach, we present a recipe for the convergence analysis of projected gradient descent and demonstrate it via a beginning-to-end application of the recipe on four fundamental problems, namely, linear equality-constrained least squares, sparse recovery, least squares with the unit norm constraint, and matrix completion.
\end{abstract}


%
\IEEEpeerreviewmaketitle

\section{Introduction}

\IEEEPARstart{C}{onstrained} least squares can be formulated as the following optimization problem: 
\begin{align} \label{prob:f}
\min_{\bm x \in \R^n} \frac{1}{2} \unorm{\bm A \bm x - \bm b}^2 \quad \text{s.t. } \bm x \in \C ,
\end{align}
where $\C \in \R^n$ is a non-empty closed set, $\bm A \in \R^{m \times n}$, and $\bm b \in \R^m$ is the observation from which we wish to recover the solution $\bm x^*$ efficiently.  
With the surge in the amount of data over the past decades, modern learning problems have become increasingly complex and optimization in the presence of constraints is frequently used to capture accurately their inherent structure. 
Examples in the area of machine learning and signal processing include, but are not limited to, compressed sensing \cite{figueiredo2007gradient,blumensath2008iterative,blumensath2009iterative}, image restoration \cite{hunt1973application,galatsanos1991least,mesarovic1995regularized}, seismic inversion \cite{puryear2012constrained,chen2015seismic,menke2018geophysical}, and phase-only beamforming \cite{tranter2017fast,zhang2021fast}. Since the set of real $n_1 \times n_2$ matrices is isomorphic to $\R^{n_1 n_2}$, application of (\ref{prob:f}) is also found in problems such as low-rank matrix recovery \cite{jain2010guaranteed,chen2015fast,khanna2018iht} and non-negative matrix factorization \cite{lin2007projected,mohammadiha2009nonnegative,guan2012nenmf}.

Projected gradient descent (PGD) is one of the most popular methods for solving constrained optimization, thanks to its simplicity and efficiency.
In theory, convergence properties of this method are natural extensions of the classical results for unconstrained optimization \cite{luenberger1984linear,bertsekas1997nonlinear,beck2017first,jain2017non}. 
When the constraint set $\C$ is convex, PGD is also known as the projected Landweber iteration \cite{combettes2011proximal} and is shown to converge sublinearly to the global solution of (\ref{prob:f}). Moreover, when the least-squares objective is strongly convex, the algorithm enjoys fast linear convergence. For non-convex settings, with the recent introduction of restricted (strong) convexity, global convergence has been guaranteed for certain structural constraints such as sparsity constraint \cite{candes2005decoding}, low-rank constraint \cite{sun2016guaranteed}, and L2-norm constraint \cite{beck2018globally}.

From a different perspective, problem (\ref{prob:f}) can be viewed as a manifold optimization problem in which the intrinsic structure of manifolds can be exploited. Dating back to the 1970s, Luenberger \cite{luenberger1972gradient} studied a variant of gradient projection method using the concept of geodesic descent. 
Under the assumption that $\C$ is a differentiable manifold in Euclidean space, the author provided sufficient conditions for global convergence and established a sharp bound on the asymptotic convergence rate near a strict local minimum. Later on, this result was extended to a broader class of Riemannian manifolds and has been widely known as the Riemannian steepest descent method\cite{luenberger1972gradient,lichnewsky1979minimisation,gabay1982minimizing,uschmajew2020geometric}.
The asymptotic convergence rate of Riemannian steepest descent (with exact line search) is given by the Kantorovich ratio $(\beta-\alpha)^2/(\beta+\alpha)^2$, where $\alpha$ and $\beta$ are the smallest and largest eigenvalues of the second derivative of the Lagrangian restricted to the subspace tangent to the constraint manifold at the solution. 
Remarkably, such local convergence bounds are tighter than those obtained from the aforementioned global convergence analysis in the optimization literature since the former exploits the local structure of the problem. 
The global convergence bounds, on the other hand, take into account the worst-case behavior of the algorithm that might occur far away from the solution of interest.
In certain situations, global convergence analysis suggests sublinear convergence while local convergence analysis offers linear convergence thanks to the benign structure near the solution \cite{o2015adaptive}. 
One key element in the asymptotic convergence analysis of Riemannian steepest descent is Kantorovich inequality \cite{strang1960kantorovich}.
However, this technique depends on the optimal choice of step size in the exact line search scheme and is not straightforwardly generalized to other variants of gradient projection.  To the best of our knowledge, there has been no direct extension of the analysis for Riemannian steepest descent method to plain PGD with a fixed step size. 

\textbf{Our Contribution.} In this paper, we develop a unified framework for a local convergence analysis of the PGD algorithm. 
We leverage our earlier preliminary work, in which we developed a convergence rate only analysis for the specific problems of low-rank matrix completion \cite{chunikhina2014performance} and minimization of a quadratic with spherical constraints \cite{vu2019convergence}. For the former, we developed two acceleration approaches that leverage on the rate analysis \cite{vu2019local,vu2019accelerating}. The key approach used in these works is to represent each algorithm as a fixed point iteration and to approximate the fixed point operator as locally linear. This idea extends to other algorithms (i.e., non PGD) that can be represented using a fixed point iteration (e.g., see our work on analyzing GD for symmetric matrix completion \cite{vu2021exact}). For each problem, problem-specific properties have been utilized to facilitate the analysis. Here, our goal is to develop a {\em unified} framework for convergence rate analysis of PGD for constrained least-squares.
Our framework relies on three key steps: {\em(i)} the introduction of Lipschitz-continuous differentiability to provide tight error bounds on the linear approximation of the projection operator near the solution, {\em(ii)}
the establishment of an asymptotically-linear recursion on the error iterations, and {\em(iii)} the derivation of the linear rate and the region of convergence (ROC) of the error sequence by leveraging our work on the convergence of nonlinear difference equations \cite{vu2021closed}. Our approach shifts the burden of the analysis to the characterization of the projection operator (for an example of such characterization of the projection onto the rank-$r$ manifold, see \cite{vu2021perturbation}-Theorem~1).
In the context of PGD for the general constrained least squares, the proposed framework is the first to offer a closed-form expression of the exact asymptotic rate of local linear convergence, the ROC, and a bound on the number of iterations needed to reach a certain level of accuracy.\footnote{We note that the classic work of Polyak \cite{polyak1964some} can be considered as a replacement for our analysis in the third step. While such result is more general in the context of nonlinear different equations, we do not find a straightforward extension to obtain the ROC and the guarantees on the number of required iterations in our context of convergence analysis.}
To illustrate the utility of the approach, we apply our framework to four well-known problems in machine learning and signal processing, namely, linear equality-constrained least squares, sparse recovery, least squares with spherical constraint, and matrix completion.
We show that the obtained asymptotic rate of convergence matches existing results in the literature. For problems in which the exact convergence rate of PGD has not been studied, we verify the asymptotic rate obtained by our analysis against the rate of convergence obtained in numerical experiments. 
We believe that this framework can be used as a general recipe to develop quick yet sharp local convergence results for PGD in other applications in the field as well as to complement conservative analysis of global convergence.

\textbf{Organization.} The rest of this paper is organized as follows. Section~\ref{sec:prel} provides a brief background of PGD for constrained least squares, including properties of the orthogonal projection, stationary points of the problem, and the PGD algorithm along with its fixed points. Next, we present our unified framework for the local convergence analysis of PGD in Section~\ref{sec:main}, followed by the proof of the main theorem.
Then, Section~\ref{sec:apps} demonstrates the application of the proposed recipe to four well-known problems in machine learning and signal processing. Finally, we summarize our results and discuss some of the possible extensions in Section~\ref{sec:conclusion}.

\section{Preliminaries}
\label{sec:prel}

This section presents key concepts and background results that will be used as the basic premise of our subsequent convergence analysis.

\subsection{Notations}
Throughout the paper, we use the notation $\unorm{\cdot}$ to denote the Euclidean norm for vectors. For matrices, $\unorm{\cdot}_F$ and $\unorm{\cdot}_2$ denote the Frobenius norm and the spectral norm, respectively. Boldfaced symbols are reserved for vectors and matrices. 
Additionally, the $t \times t$ identity matrix is denoted by $\bm I_t$ and the $i$th vector in the natural basis of $\R^n$ is denoted by $\bm e_i$. We use $\otimes$ to denote the Kronecker product between two matrices. 
The vectorization of a matrix $\bm X \in \R^{m \times n}$, denoted by $\vect(\bm X)$, is the concatenation of the columns of a matrix one on top of another in their original order, i.e., for $\bm X=[{\bm x}_1,\ldots,{\bm x}_n] $, $\vect(\bm X)=[{\bm x}_1^\topnew,\ldots,{\bm x}_n^\topnew]^\topnew$. 
Given a vector $\bm x \in \R^n$, $\diag(\bm x)$ denotes the a square diagonal matrix such that $[\diag(\bm x)]_{ii}=x_i$.
For a scalar $r>0$, denote the open ball of center $\bm x$ and radius $r$ by $\B(\bm x,r) = \{ \bm y \mid \unorm{\bm y - \bm x} < r \}$. 
Correspondingly, the closed ball of center $\bm x$ and radius $r$ is denoted by $\B[\bm x,r] = \{ \bm y \mid \unorm{\bm y - \bm x} \leq r \}$.
The lexicographical order between two vectors $\bm x$ and $\bm y$ of the same length is defined by $\bm x < \bm y$ if $x_i < y_i$ for the first $i$ ($i$ goes from $1$) where $x_i$ and $y_i$ differ. The lexicographical order between two matrices $\bm X$ and $\bm Y$ of the same size is define by the lexicographical order between $\vect(\bm X)$ and $\vect(\bm Y)$. 

Given $\bm A \in \R^{m \times n}$, the $i$th largest eigenvalue and the $i$th largest singular value of $\bm A$ are denoted by $\lambda_i(\bm A)$ and $\sigma_i(\bm A)$, respectively.
The spectral radius of $\bm A$ is defined as $\rho(\bm A) = \max_{i} \uabs{\lambda_i(\bm A)}$ and is less than or equal to the spectral norm, i.e., $\rho(\bm A) \leq \unorm{\bm A}_2$. 
Gelfand's formula \cite{gelfand1941normierte} states that $\rho(\bm A) = \lim_{k \to \infty} \unorm{\bm A^k}_2^{1/k}$. 
If $\bm A$ is square and invertible, the condition number of $\bm A$ is defined as $\kappa(\bm A) = \sigma_1(\bm A)/\sigma_n(\bm A)$.

\subsection{Nonlinear Orthogonal Projections}

Given a non-empty set $\C \subset \R^n$, let us define the distance from a point $\bm x \in \R^n$ to $\C$ as
\begin{align}
    d(\bm x, \C) = \inf_{\bm y \in \C} \{ \unorm{\bm y - \bm x} \} .
\end{align}
The set of all projections of $\bm x$ onto $\C$ is defined by
\begin{align} \label{equ:PC}
    \Pi_\C (\bm x) = \{ \bm y \in \C \mid \unorm{\bm y - \bm x} = d(\bm x, \C) \} .
\end{align}
It is well-known \cite{vasilyev2013depth} that if $\C$ is closed, then for any $\bm x \in \R^n$, $\Pi_\C (\bm x)$ is non-empty\footnote{In addition, if $\C$ is convex, then $\Pi_\C (\bm x)$ is singleton.}.
An orthogonal projection onto $\C$ is defined as $\P_\C: \R^n \to \C$ such that $\P_\C (\bm x)$ is chosen as an element of $\Pi_\C (\bm x)$ based on a prescribed scheme (e.g., based on lexicographic order).
There exists a non-empty subset of $\R^n$ such that $\Pi_\C$ is uniquely defined, given by
\begin{align} \label{equ:uni_Pc}
    \sing \Pi_\C = \{ \bm x \in \R^n \mid \Pi_\C (\bm x) \text{ is singleton} \} .
\end{align}
We can now consider the differentiability of $\P_\C$ over $\sing \Pi_\C$ as follows.
\begin{definition}[Point-wise differentiability] 
\label{def:Pc_diff}
The projection $\P_\C$ is said to be \textbf{differentiable} at $\bm x \in \sing \Pi_\C$ if there exists $\nabla \P_\C (\bm x) \in \R^{n \times n}$ such that
\begin{align*}
    \limsup_{\bm \delta \to \bm 0} \sup_{\bm y \in \Pi_\C (\bm x + \bm \delta) } \frac{\unorm{\bm y - \P_\C(\bm x) - \nabla \P_\C (\bm x) \bm \delta}}{\unorm{\bm \delta}} = 0 .
\end{align*}
The operator $\nabla \P_\C (\bm x)$ is said to be the derivative of $\P_\C$ at $\bm x$.
\end{definition}
\begin{definition}[Point-wise Lipschitz-continuous differentiability]
\label{def:Pc_diff2}
The projection $\P_\C$ is said to be Lipschitz-continuously differentiable at $\bm x \in \sing \Pi_\C$ if $\P_\C$ is differentiable at $\bm x$ and there exist $0 < c_1(\bm x) \leq \infty$ and $0 \leq c_2(\bm x) < \infty$ such that for any $\bm \delta \in \B(\bm 0,c_1(\bm x))$, we have
\begin{align} \label{equ:sup_Pc}
    \sup_{\bm y \in \Pi_\C (\bm x + \bm \delta) } \unorm{\bm y - \P_\C(\bm x) - \nabla \P_\C (\bm x) \bm \delta} \leq c_2(\bm x) \unorm{\bm \delta}^2 .
\end{align}
\end{definition}
\noindent It is noted that the supremum in (4) implies 
\begin{align*}
    \unorm{\P_\C(\bm x + \bm \delta) - \P_\C(\bm x) - \nabla \P_\C (\bm x) \bm \delta} \leq c_2(\bm x) \unorm{\bm \delta}^2
\end{align*}
holds for any choice of $\P_\C(\bm x + \bm \delta)$ in $\Pi_{\cal C}({\bm x}+{\bm \delta})$.
Note that while $\P_\C(\bm x)$ is uniquely defined for $\bm x \in \sing \Pi_\C$, $\P_\C(\bm x + \bm \delta)$ is not since $\bm x + \bm \delta$ may not be in $\sing \Pi_\C$.
\begin{example} \label{eg:c2_Ps}
Let $\C = \{ \bm x \in \R^n \mid \unorm{\bm x}=1 \}$ be the unit sphere of dimension $n-1$. For any $\bm x \neq \bm 0$, the projection onto $\C$ is uniquely given by $\P_\C (\bm x) = \bm x / \unorm{\bm x}$.
For $\bm x = \bm 0$, we have $\Pi_{\C}(\bm 0) = \C$ and $\P_\C(\bm 0)$ can be chosen as any point on the unit sphere.
In Section~II of the Supplementary Material, we prove that $\P_\C$ is Lipschitz-continuously differentiable at any $\bm x \in \sing \Pi_\C = \R^n \setminus \{\bm 0\}$. In particular, for any $\bm x \neq \bm 0$ and $\bm \delta \in \R^n$, we have
\begin{align*} 
    \sup_{\bm y \in \Pi_{\C}(\bm x + \bm \delta)} \norm{\bm y - \frac{\bm x}{\unorm{\bm x}} - \Bigl( \bm I_n - \frac{\bm x \bm x^{\topnew}}{\unorm{\bm x}^2} \Bigr) \frac{\bm \delta}{\unorm{\bm x}}} \leq \frac{2}{\unorm{\bm x}^2} \unorm{\bm \delta}^2 . \numberthis \label{equ:sphere_off}
\end{align*}
For $\bm \delta \neq - \bm x$, $\Pi_\C(\bm x + \bm \delta) = \{ (\bm x + \bm \delta) / \norm{\bm x + \bm \delta} \}$ is singleton and the supremum is evaluated at only one point $\bm y = (\bm x + \bm \delta) / \norm{\bm x + \bm \delta}$.
For $\bm \delta = - \bm x$, $\Pi_\C(\bm x + \bm \delta) = \Pi_\C(\bm 0) = \C$ is not singleton and the supremum is taken over the entire sphere independent of $\bm x$.
In either case regardless the value of $\bm \delta$, comparing (\ref{equ:sphere_off}) with (\ref{equ:sup_Pc}), we recognize the projection onto the unit sphere is Lipschitz-continuously differentiable at $\bm x \in \sing \Pi_\C$ with
\begin{align*}
    &\nabla \P_\C(\bm x) = \frac{1}{\unorm{\bm x}} \Bigl( \bm I_n - \frac{\bm x \bm x^{\topnew}}{\unorm{\bm x}^2} \Bigr) , \\
    &c_1(\bm x)=\infty , \quad c_2(\bm x)=\frac{2}{\unorm{\bm x}^2} .
\end{align*}
\end{example}
In 1984, Foote \cite{foote1984regularity} showed that if $\C$ is a $C^k$ ($k \geq 2$) submanifold of $\R^n$, then $\C$ has a neighborhood $\mathcal{E}$ such that $\mathcal{E} \subseteq \sing \Pi_\C$ and the projection $\P_\C$ restricted to $\mathcal{E}$ is a $C^{k-1}$ mapping.
Later on, Dudek and Holly \cite{dudek1994nonlinear} proved the derivative $\nabla \P_\C$ is a linear map to the tangent bundle of $\C$ and more importantly, for any $\bm x^* \in \C$, $\nabla \P_\C (\bm x^*)$ is the (linear) orthogonal projection onto the tangent space to $\C$ at $\bm x^*$.
Recently, a local version of this result has been proposed by Lewis and Malick \cite{lewis2008alternating}:

\begin{proposition} \label{prop:pc} (Rephrased from Lemma~4 in \cite{lewis2008alternating})
Assume $\C$ is a $C^k$ ($k \geq 2$) manifold around a point $\bm x^* \in \C$. Denote the tangent space to $\C$ at $\bm x^*$ by $T_{\bm x^*}(\C)$. Then, the set of projections $\Pi_\C$ is (locally) singleton around $\bm x^*$. Moreover, $\P_\C$ is a $C^{k-1}$ mapping around $\bm x^*$ and
\begin{align} \label{equ:dudek0}
    \nabla \P_\C (\bm x^*) = \P_{T_{\bm x^*}(\C)} ,
\end{align}
where $\P_{T_{\bm x^*}(\C)}$ is the orthogonal projection onto $T_{\bm x^*}(\C)$.
\end{proposition}

\noindent Further works on the uniqueness and regularity of $\P_\C$ can also be found in \cite{ambrosio1998curvature,absil2012projection,rataj2019curvature,leobacher2021existence}.
We note that the assumption $\C$ is a $C^2$ manifold around $\bm x^*$ requires the existence of a neighborhood of $\bm x^*$ in which $\P_\C$ is uniformly differentiable. 
Our result in this manuscript, while strongly motivated by the aforementioned results, only requires $\C$ to be differentiable at two points (see Theorem~\ref{theo:pgd}).

\begin{algorithm}[t]
\caption{Projected Gradient Descent (PGD)}
\label{algo:PGD}
\textbf{Input}: $f$, $\C$, $\eta$, $\bm x^{(0)}$ \\
\textbf{Output}: $\{\bm x^{(k)}\}_{k=0}^\infty$
\begin{algorithmic}[1]
\For{$k=0,1,\ldots$}
\State $\bm z^{(k)}_\eta = \bm x^{(k)} - \eta \bm A^\topnew\bigl( \bm A \bm x^{(k)} - \bm b \bigr)$
\State $\bm x^{(k+1)} = \P_{\C} \bigl( \bm z^{(k)}_\eta \bigr)$
\EndFor
\end{algorithmic}
\end{algorithm}

\subsection{Stationary Points of (\ref{prob:f})}

We defined the (Lipschitz-continuous) differentiability of the projection $\P_\C$ at a point in $\C$.
We are now in position to define stationary points of (\ref{prob:f}) as those where the gradient of the objective function on the constraint set vanishes \cite{absil2009optimization}:
\begin{definition}\label{def:stationary}
$\bm x^* \in \C$ is a \textbf{stationary point} of (\ref{prob:f}) if $\P_\C$ is differentiable at $\bm x^*$ and
\begin{align} \label{equ:stationary}
    \nabla \P_\C (\bm x^*) \bm A^\topnew\bigl( \bm A \bm x^* - \bm b \bigr) = \bm 0 .
\end{align}
Assume in addition that $\P_\C$ is Lipschitz-continuously differentiable at $\bm x^*$ with constants $c_1(\bm x^*)$ and $c_2(\bm x^*)$. Then $\bm x^*$ is called a \textbf{Lipschitz stationary point} of (\ref{prob:f}) with constants $c_1(\bm x^*)$ and $c_2(\bm x^*)$.
\end{definition}
\noindent Similar to unconstrained optimization, stationary points in Definition~\ref{def:stationary} can be local minimizers, local maximizers, or saddle points of the constrained problem (\ref{prob:f}).

\subsection{Projected Gradient Descent}

Algorithm~\ref{algo:PGD} describes the projected gradient descent algorithm for solving (\ref{prob:f}). Starting at some $\bm x^{(0)} \in \C$, the algorithm iteratively updates the current value by {\em(i)} taking a step in the opposite direction of the gradient and {\em(ii)} projecting the result back onto $\C$, i.e.,
\begin{align} \label{equ:PGD}
    \bm x^{(k+1)} = \P_{\C} \Bigl( \bm x^{(k)} - \eta \bm A^\topnew\bigl( \bm A \bm x^{(k)} - \bm b \bigr) \Bigr) ,
\end{align}
where $\eta>0$ is a fixed step size.

\begin{definition}
$\bm x^*$ is a \textbf{fixed point} of Algorithm~\ref{algo:PGD} with step size $\eta>0$ if 
\begin{align} \label{equ:fixed_point}
    \bm x^* = \P_{\C} (\bm x^* - \eta \bm A^\topnew ( \bm A \bm x^* - \bm b )) .
\end{align}
\end{definition}

\begin{lemma} \label{lem:fixed}
If $\bm x^*$ is a fixed point of Algorithm~\ref{algo:PGD} with some step size $\eta>0$ and $\P_\C$ is differentiable at $\bm x^*$, then $\bm x^*$ is a stationary point of (\ref{prob:f}).
\end{lemma}
\noindent The proof of Lemma~\ref{lem:fixed} is given in Appendix~\ref{appdx:fixed}.

\section{Local Convergence Analysis}
\label{sec:main}

\begin{figure}[t]
    \centering
    \includegraphics[width=\columnwidth]{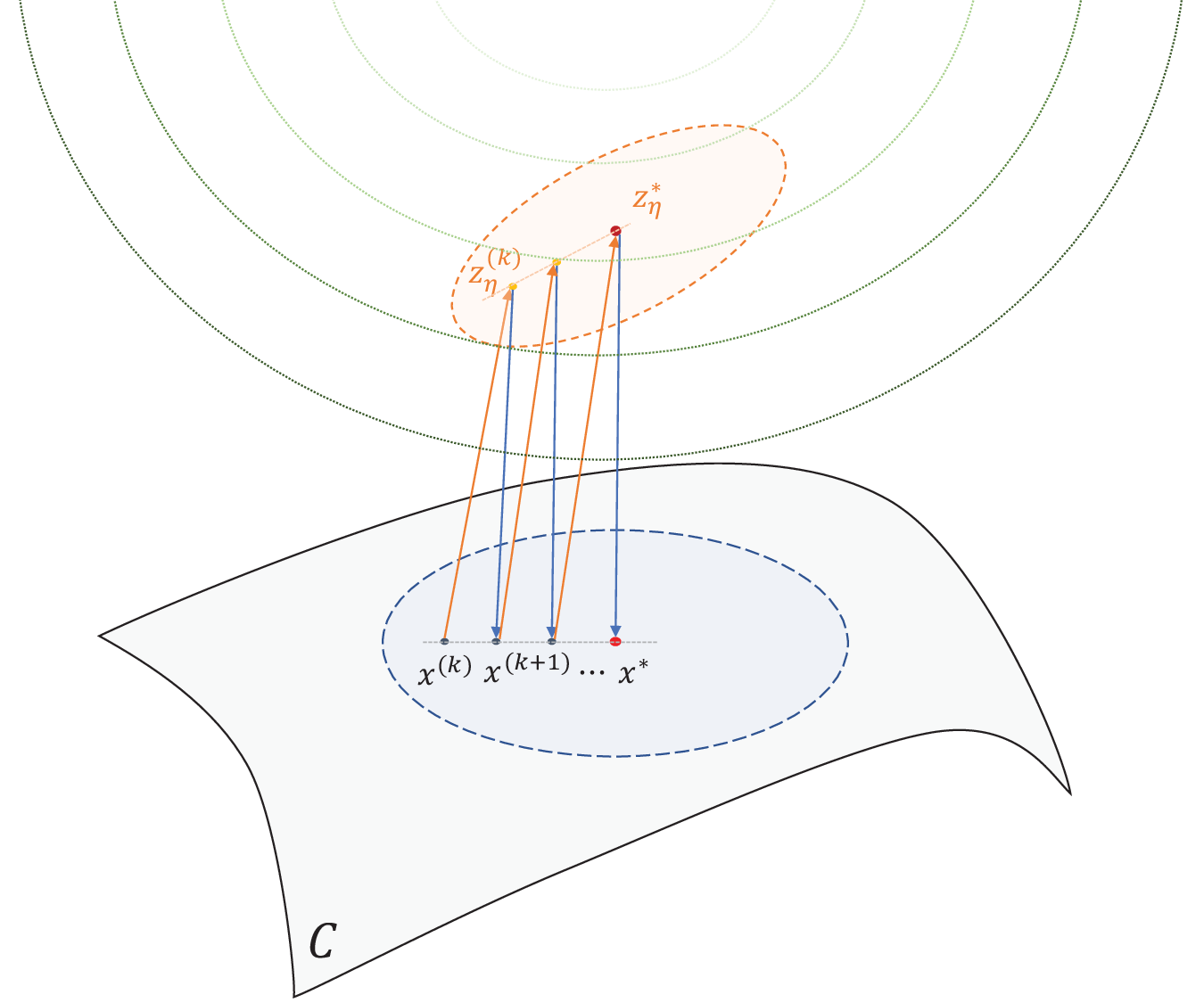}
    \caption{Illustration of convergence of projected gradient descent to a fixed point $\bm x^*$. In order to guarantee linear convergence, Theorem~\ref{theo:pgd} requires $\P_\C$ to be Lipschitz-continuously differentiable at both $\bm x^{(k)}$ and $\bm z_\eta^{(k)} = \bm x^{(k)} - \eta \bm A^\topnew (\bm A \bm x^{(k)} - \bm b)$. Moreover, the condition $\unorm{\bm x^{(0)} - \bm x^*} < \min \{ {c_1(\bm x^*)}/{\kappa(\bm Q)}, {c_1(\bm z^*_\eta)}/{(\kappa(\bm Q) u_\eta)} \}$ from (\ref{cond:roc}) ensures that $\bm x^{(k)}$ remains inside $\B(\bm x^*,c_1(\bm x^*))$ (blue dashed ellipse) and $\bm z_\eta^{(k)}=$ remains inside $\B(\bm z^*_\eta,c_1(\bm z^*_\eta))$ (orange dashed ellipse).}
    \label{fig:xz}
\end{figure}

In this section, we present the key contribution of this work, namely, a local convergence analysis of projected gradient descent for constrained least squares. Specifically, \textbf{our goal} is to establish the following results: {\em(i)} a closed-form expression of the exact asymptotic rate of convergence, {\em(ii)} a bound on the number of iterations needed to reach a certain level of accuracy, and {\em(iii)} a region of convergence.
Figure~\ref{fig:xz} illustrates the key idea in our analysis.
In order to establish the local linear convergence of Algorithm~\ref{algo:PGD} to its fixed point $\bm x^*$, we require the Lipschitz-continuous differentiability of $\P_\C$ at $\bm x^*$ and at $\bm z^*_\eta = \bm x^* - \eta \bm A^\topnew ( \bm A \bm x^*  - \bm b )$. These properties enables us to approximate each projected gradient descent update by a linear operator on the error vector (i.e., the difference between $\bm x^{(k)}$ and $\bm x^*$). 
Then, under the additional assumption that this linear operator is a contraction mapping and the initialization $\bm x^{(0)}$ is sufficiently close to $\bm x^*$, we show that the gradient step and the projection step remain inside the Lipschitz-continuous differentiability regions of $\bm x^*$ (i.e., $\B(\bm x^*,c_1(\bm x^*))$) and $\bm z^*_\eta$ (i.e., $\B(\bm z^*_\eta,c_1(\bm z^*_\eta))$, respectively).

\subsection{Main Results}

In this following, we state our main result in Theorem~\ref{theo:pgd}, followed by further insights into the convergence results. 

\begin{theorem} \label{theo:pgd}
Suppose $\bm x^*$ is a fixed point of Algorithm~\ref{algo:PGD} with step size $\eta>0$ such that the following conditions hold:
\begin{enumerate}[leftmargin=*]
    \item $\P_\C$ is Lipschitz-continuously differentiable at both the fixed point $\bm x^*$ and at the gradient step taken from the fixed point
    \begin{align} \label{equ:z}
        \bm z^*_\eta = \bm x^* - \eta \bm A^\topnew\bigl( \bm A \bm x^*  - \bm b \bigr) ,
    \end{align}
    with the corresponding matrices $\nabla \P_{\C} (\bm x^*)$, $\nabla \P_{\C} (\bm z^*_\eta)$, and constants $c_1(\bm x^*)$, $\bm c_2(\bm x^*)$, $c_1(\bm z^*_\eta)$, and $c_2(\bm z^*_\eta)$.
    
    \item The matrix
    \begin{align} \label{equ:H}
        \bm H = \nabla \P_{\C} (\bm z^*_\eta) (\bm I_n - \eta \bm A^\topnew \bm A) \nabla \P_{\C} (\bm x^*) 
    \end{align}
    admits an eigendecomposition $\bm H = \bm Q \bm \Lambda \bm Q^{-1}$, where $\bm Q \in \R^{n \times n}$ is an invertible matrix and $\bm \Lambda$ is a diagonal matrix whose diagonal entries are strictly less than $1$ in magnitude, i.e., $\rho(\bm H) = \unorm{\bm \Lambda}_2 < 1$.

    \item The initial iterate $\bm x^{(0)}$ satisfies
    \begin{align*}
        &\unorm{\bm x^{(0)} - \bm x^*} < \min \Bigl\{ \frac{c_1(\bm x^*)}{\kappa(\bm Q)}, \frac{c_1(\bm z^*_\eta)}{\kappa(\bm Q) u_\eta}, \frac{1-\rho(\bm H)}{q} \Bigr\} , \numberthis \label{cond:roc}
    \end{align*}
    where 
    \begin{align} \label{def:u_eta}
        u_\eta = \unorm{\bm I_n - \eta \bm A^\topnew \bm A}_2
    \end{align}
    and
    \begin{align} \label{def:q}
        q = \kappa^2(\bm Q) u_\eta \bigl( c_2(\bm z^*_\eta) u_\eta + \unorm{\nabla \P_{\C} (\bm z^*_\eta)}_2 c_2(\bm x^*) \bigr) .
    \end{align} 
\end{enumerate}
Let $\{\bm x^{(k)}\}_{k=0}^\infty$ be the vector sequence generated by the PGD update in (\ref{equ:PGD}).
Then, for any $0<\epsilon<1$, we have $\unorm{\bm x^{(k)} - \bm x^*} \leq \epsilon \unorm{\bm x^{(0)} - \bm x^*}$ for all
\begin{align} \label{equ:k}
    k \geq \frac{\log(1/\epsilon)+\log(\kappa(\bm Q))}{\log(1/\rho(\bm H))} + c_3 ,
\end{align}
where $c_3>0$, given explicitly in Lemma~\ref{lem:scalar}, is independent of $\epsilon$.
Algorithm~\ref{algo:PGD} is said to converge locally to $\bm x^*$ at an \textbf{asymptotic linear rate} $\rho(\bm H)$ with the \textbf{region of linear convergence} given by (\ref{cond:roc}).
\end{theorem}

\noindent Theorem~\ref{theo:pgd} states the sufficient conditions for asymptotic linear convergence of Algorithm~\ref{algo:PGD}. 
In addition, the theorem establishes the asymptotic rate as the spectral radius of the matrix $\bm H$ and bounds the number of iterations needed to reach $\epsilon$-accuracy.
The proof of Theorem~\ref{theo:pgd} is given in Subsection~\ref{subsec:proof1}.
It is noteworthy that in the RHS of (\ref{equ:k}), the first term corresponds to linear convergence in the asymptotic regime and the second term corresponds to nonlinear convergence behavior at the early stage. We will revisit this point when we introduce Lemma~\ref{lem:scalar}.

\begin{remark} \label{rmk:kappa}
When $\bm H$ is symmetric, its eigendecomposition exists and can be represented as
\begin{align*}
    \bm H = \bm Q \bm \Lambda \bm Q^\topnew ,
\end{align*}
where $\bm Q$ is an orthogonal matrix with $\kappa(\bm Q)=1$. 
\end{remark}
Next, we study a special case of Theorem~\ref{theo:pgd} in which
\begin{align} \label{equ:PZT}
    \nabla \P_{\C} (\bm z^*_\eta) = \nabla \P_\C (\bm x^*) = \P_{T_{\bm x^*}(\C)} = \bm U_{\bm x^*} \bm U_{\bm x^*}^\topnew ,
\end{align}
where $\bm U_{\bm x^*} \in \R^{n \times d}$ ($d \leq n$) is the matrix whose columns provide an orthonormal basis for the tangent space to $\C$ at $\bm x^*$.
A typical example in which (\ref{equ:PZT}) holds is when {\em(i)} $\C$ is a $C^2$ $d$-dimensional submanifold around $\bm x^*$; and {\em(ii)} $\bm z^*_\eta = \bm x^*$. The first condition {\em(i)} stems from Proposition~\ref{prop:pc} in order to guarantee $\nabla \P_\C (\bm x^*) = \P_{T_{\bm x^*}(\C)}$. The second condition {\em(ii)} is equivalent to $\bm A^\topnew (\bm A \bm x - \bm b) = \bm 0$, which means $\bm x^*$ is also a stationary point of the unconstrained problem. Conveniently, this coincidence eliminates the task of characterizing the projection $\P_\C$ and its derivative $\nabla \P_\C$ at a point outside $\C$, which can be a challenging task in many problems.

\begin{corollary} \label{cor:noiseless}
Consider the same setting as in Theorem~\ref{theo:pgd} with the additional assumption that (\ref{equ:PZT}) holds.
If $(\bm A \bm U_{\bm x^*})^\topnew \bm A \bm U_{\bm x^*}$ has full rank and
\begin{align} \label{equ:eta2}
    0 < \eta < \frac{2}{\unorm{\bm A \bm U_{\bm x^*}}_2^2} ,
\end{align}
then Algorithm~\ref{algo:PGD} with fixed step size $\eta$ converges locally to $\bm x^*$ at an asymptotic linear rate \begin{align} \label{equ:rho_noiseless}
    \rho(\bm H) = \max \{ \uabs{1-\eta \lambda_1} , \uabs{1-\eta \lambda_d} \} ,
\end{align} 
where $\lambda_1$ and $\lambda_d$ are the largest and smallest eigenvalues of $(\bm A \bm U_{\bm x^*})^\topnew \bm A \bm U_{\bm x^*}$, respectively.
The region of linear convergence is given by
\begin{align} \label{cond:roc_sym}
    \unorm{\bm x^{(0)} - \bm x^*} < \min \Bigl\{ c_1(\bm x^*), \frac{c_1(\bm z^*_\eta)}{u_\eta}, \frac{1-\rho(\bm H)}{u_\eta c_2(\bm x^*) + u_\eta^2 c_2(\bm z^*_\eta)} \Bigr\} ,
\end{align}
where $u_\eta$ is given by (\ref{def:u_eta}).
\end{corollary}
\noindent The proof of Corollary~\ref{cor:noiseless} is given in Appendix~\ref{appdx:noiseless}. 

\begin{remark} \label{rmk:interlace}
Recall that $u_\eta$ defined in  (\ref{def:u_eta}) is also the asymptotic linear rate of gradient descent for the unconstrained least squares \cite{polyak1987introduction}, i.e.,
\begin{align*}
    u_\eta = \max \{ \uabs{1-\eta \lambda_1(\bm A^\topnew \bm A)} , \uabs{1-\eta \lambda_n(\bm A^\topnew \bm A)} \} .
\end{align*}
Since $\bm U_{\bm x^*}$ is a semi-orthogonal matrix, the eigenvalues of $\bm U_{\bm x^*}^\topnew \bm A^\topnew \bm A \bm U_{\bm x^*}$ interlace with those of $\bm A^{\topnew} \bm A$ \cite{hwang2004cauchy}, which in turns implies $\lambda_n(\bm A^\topnew \bm A) \leq \lambda_d \leq \lambda_1 \leq \lambda_1(\bm A^\topnew \bm A)$.
Thus, one can show that for $\eta<2/\unorm{\bm A}_2^2$,
\begin{align} \label{equ:rho_unconstrained}
    \rho(\bm H) \leq u_\eta \leq 1 ,
\end{align}
with the equality $u_\eta = 1$ holding if and only if $\bm A^\topnew \bm A$ is singular. Interestingly, (\ref{equ:rho_unconstrained}) implies the presence of the constraint in this case helps accelerate the convergence of gradient descent to $\bm x^*$.
\end{remark}

\subsection{Proof of Theorem~\ref{theo:pgd}}
\label{subsec:proof1}

This section presents the proof of Theorem~\ref{theo:pgd}. 
Our key ideas are: {\em(1)} using the Lipschitz-continuous differentiability of $\P_\C$ at $\bm x^*$ and at $\bm z_\eta^*$ to establish a recursive relation on the error vector $\bm \delta^{(k)} = \bm x^{(k)} - \bm x^*$, {\em(2)} performing a change of basis $\tilde{\bm \delta}^{(k)} = \bm Q^{-1} \bm \delta^{(k)}$ to establish an asymptotically-linear quadratic system dynamic that upper-bounds the norm of the transformed error vector, {\em(3)} applying the result on the convergence of an asymptotically-linear quadratic difference equation in \cite{vu2021closed} to obtain the number of iterations required for $\unorm{\tilde{\bm \delta}^{(k)}} \leq \tilde{\epsilon} \unorm{\tilde{\bm \delta}^{(0)}}$, and {\em(4)} converting the convergence result on the transformed error $\unorm{\tilde{\bm \delta}^{(k)}}$ to the convergence result on the original error $\unorm{{\bm \delta}^{(k)}}$. 
In the following, we provide the complete proof, with some details deferred to the appendix. \\[1pt]

\noindent \textbf{Step 1:} 
Let us define the error vector of Algorithm~\ref{algo:PGD} as $\bm \delta^{(k)} = \bm x^{(k)} - \bm x^*$, for $k \in \mathbb{N}$.
Using this definition of the error vector, we can replace $ \bm x^{(k)} = \bm x^* + \bm \delta^{(k)} $ and $ \bm x^{(k+1)} = \bm x^* + \bm \delta^{(k+1)} $ into (\ref{equ:PGD}) to obtain an equivalent update on the error vector
\begin{align}
    \bm \delta^{(k+1)} &= \P_{\C} \Bigl( \bm x^* + \bm \delta^{(k)} - \eta \bm A^\topnew\bigl( \bm A(\bm x^* + \bm \delta^{(k)}) - \bm b \bigr) \Bigr) - \bm x^*.  \label{equ:delta_1}
\end{align}
Based on the definition 
of $\bm z^*_\eta $ in (\ref{equ:z}) and the
fact that $\bm x^*$ is a fixed point of the algorithm (see (\ref{equ:fixed_point})), i.e.,  $\bm x^* = \P_\C ( \bm z^*_\eta )$, we can rewrite (\ref{equ:delta_1}) as 
\begin{align*} 
    \bm \delta^{(k+1)} &= \P_\C \Bigl( \bm z^*_\eta + (\bm I - \eta \bm A^\topnew \bm A) \bm \delta^{(k)} \Bigr) - \P_{\C} (\bm z^*_\eta) . \numberthis \label{equ:delta_2}
\end{align*}
We are now in position to analyze the error update as a fixed-point iteration: $\bm \delta^{(k+1)} = \bm f( \bm \delta^{(k)})$, where $\bm f(\bm \delta) = \P_\C ( \bm z^*_\eta + (\bm I - \eta \bm A^\topnew \bm A) \bm \delta ) - \P_{\C} (\bm z^*_\eta)$.
The following lemma provides a recursive equation on the error vector that is in the form of an asymptotically-linear quadratic system dynamic:
\begin{lemma} \label{lem:delta_H}
Recall $\bm H = \nabla \P_{\C} (\bm z^*_\eta) (\bm I_n - \eta \bm A^\topnew \bm A) \nabla \P_{\C} (\bm x^*)$.
If the error vector at the $k$-th iteration satisfies
\begin{align} \label{cond:delta_2}
    \unorm{\bm \delta^{(k)}} < \min \Bigl\{ {c_1(\bm x^*)}, \frac{c_1(\bm z^*_\eta)}{u_\eta} \Bigr\} ,
\end{align}
then the error vector at the $k+1$-th iteration satisfies
\begin{align} \label{equ:delta_H}
    \bm \delta^{(k+1)} = \bm H \bm \delta^{(k)} + \bm q_2(\bm \delta^{(k)}) ,
\end{align}
where $\bm q_2: \R^n \to \R^n$ is the residual such that
\begin{align} \label{equ:q_hat}
    \unorm{\bm q_2(\bm \delta^{(k)})} \leq \bigl( \unorm{\nabla \P_{\C} (\bm z^*_\eta)}_2 u_\eta c_2(\bm x^*) + c_2(\bm z^*_\eta) u_\eta^2 \bigr) \unorm{\bm \delta^{(k)}}^2 .
\end{align}
\end{lemma}
\noindent The proof of Lemma~2 is given in Appendix~\ref{appdx:delta_H}. Given the nonlinear difference equation of form (\ref{equ:delta_H}), we proceed with characterizing the convergence of the error sequence $\{ \bm \delta^{(k)} \}_{k=0}^\infty$.
\begin{remark} \label{rmk:polyak}
Dating back to 1964, Polyak \cite{polyak1964some} studied the convergence of nonlinear difference equations of form 
\begin{align} \label{equ:a_o}
    \bm a^{(k+1)} = \bm T (\bm a^{(k)}) + \bm q(\bm a^{(k)}) , \quad \text{ for } k \in \mathbb{N} ,
\end{align}
where $\bm a^{(0)} \in \R^n$, $\bm T: \R^n \to \R^n$ is a linear operator, and $\bm q: \R^n \to \R^n$ satisfies $\lim_{t \to 0} \sup_{\norm{\bm a} \leq t} \norm{\bm q(\bm a)}/\norm{\bm a} = 0$.
The author showed that if the operator $\bm T$ satisfies $\unorm{\bm T^k}_2 \leq c(\zeta) (\rho+\zeta)^k$, for some $\rho<1$ and arbitrarily small $\zeta>0$, then $\{ \bm a^{(k)} \}_{k=0}^\infty$ approaches zero with sufficiently small $\unorm{\bm a^{(0)}}$:
\begin{align} \label{equ:polyak}
    \unorm{\bm a^{(k)}} \leq C(\zeta) \unorm{\bm a^{(0)}} (\rho+\zeta)^k .
\end{align}
Here $c(\zeta)$ and $C(\zeta)$ are unknown constants that could grow to infinity as $\zeta \to 0$.
Applying this result to (\ref{equ:delta_H}) with $\bm a^{(k)} = \bm \delta^{(k)}$ and $\bm T=\bm H$, one can show that the error vector of Algorithm~\ref{algo:PGD} converges to $\bm 0$ with the asymptotic linear rate $\rho(\bm H)$, provided that $\rho(\bm H)<1$ and $\unorm{\bm \delta^{(0)}}$ is sufficiently small. 
However, we note that the proof of (\ref{equ:polyak}) in \cite{polyak1964some} is adapted from a more general result on the stability of differential equations in \cite{bellman1953stability}. This technique can not provide the precise control of the ROC and the number of iterations required to reach a certain accuracy (i.e., how small $\unorm{\bm a^{(0)}}$ is as well as how large the factor $C(\zeta)$ is) needed for our convergence analysis of PGD. Alternatively, we utilize our previous result in \cite{vu2021closed} that eliminates the dependence on $\zeta$ in the expression of the linear rate, at the cost of an additional assumption on the diagonalizability of $\bm H$.\footnote{In particular, the bound in (\ref{equ:k}) suggests $\unorm{\bm a^{(k)}} \leq C \unorm{\bm a^{(0)}} \rho^k$, for constant $C=\rho \kappa (\bm Q) e^{c_3}$, which is tighter than (\ref{equ:polyak}).} Additionally, our approach offers explicit expressions of the ROC and the number of required iterations (as in (\ref{cond:roc}) and (\ref{equ:k}), respectively).
\end{remark}

\noindent \textbf{Step 2:} 
Our approach for analyzing the convergence of the nonlinear difference equation (24) is to leverage the eigendecomposition $\bm H = \bm Q \bm \Lambda \bm Q^{-1}$ and consider the transformed error vector as follows.
\begin{lemma} \label{lem:delta_tilde}
Let $\tilde{\bm \delta}^{(k)} = \bm Q^{-1} \bm \delta^{(k)}$ be the transformed error vector.
If (\ref{cond:roc}) holds and the spectral radius of $\bm H$ is strictly less than $1$, i.e., $\rho(\bm H)<1$, then, for all $k \in \mathbb{N}$, we have
\begin{align} \label{equ:delta_tilde}
    \tilde{\bm \delta}^{(k+1)} = \bm \Lambda \tilde{\bm \delta}^{(k)} + \bm q_3 (\tilde{\bm \delta}^{(k)}) ,
\end{align}
where the residual $\bm q_3 (\tilde{\bm \delta}^{(k)}) = \bm Q^{-1} \bm q_2 (\bm Q \tilde{\bm \delta}^{(k)})$ satisfies $\unorm{\bm q_3(\tilde{\bm \delta}^{(k)})} \leq (q/\unorm{\bm Q^{-1}}_2) \unorm{\tilde{\bm \delta}^{(k)}}^2$ 
for $q$ given in (\ref{def:q}).
\end{lemma}
\noindent The proof of Lemma~\ref{lem:delta_tilde} is given in Appendix~\ref{appdx:delta_tilde}.
Taking the norms of both sides of (\ref{equ:delta_tilde}) and applying the triangle inequality, we obtain 
\begin{align*}
    &\unorm{\tilde{\bm \delta}^{(k+1)}} \leq \rho(\bm H) \unorm{\tilde{\bm \delta}^{(k)}} + \frac{q}{\unorm{\bm Q^{-1}}_2} \unorm{\tilde{\bm \delta}^{(k)}}^2 . \numberthis \label{equ:norm_delta_tilde}
\end{align*}
This inequality, holding for all $k \in \mathbb{N}$, is the key to the convergence of the transformed error sequence in the next step.

\noindent \textbf{Step 3:} 
If we replace the inequality symbol in (\ref{equ:norm_delta_tilde}) by the equality symbol, then we obtain an asymptotically-linear quadratic difference equation whose convergence is studied in \cite{vu2021closed}. 
Indeed, the following lemma states that the norm of the transformed error vector is governed by this asymptotically-linear quadratic system dynamic:
\begin{lemma} \label{lem:scalar}
Assume the same setting as Lemma~\ref{lem:delta_tilde}.
Then, for any desired accuracy  $0<\tilde{\epsilon}<1$, we have $\unorm{\tilde{\bm \delta}^{(k)}} \leq \tilde{\epsilon} \unorm{\tilde{\bm \delta}^{(0)}}$ for all
\begin{align} \label{equ:k_tilde}
    k \geq \frac{\log(1/\tilde{\epsilon})}{\log(1/\rho(\bm H))} + c_3(\rho(\bm H), \tau) ,
\end{align}
where $\tau = q \unorm{\tilde{\bm \delta}^{(0)}} / \unorm{\bm Q^{-1}}_2 / (1-\rho(\bm H)) \in (0,1)$ and
\begin{align*}
    c_3 (\rho,\tau) =~ &\frac{E_1\Bigl(\log\frac{1}{\rho+\tau(1-\rho)}\Bigr) - E_1\Bigl(\log\frac{1}{\rho}\Bigr)}{\rho \log(1/\rho)}  \\
    &+ \frac{1}{2\rho} \log \biggl( \frac{\log(1/\rho)}{\log \bigl(1/(\rho+\tau(1-\rho))\bigr)} \biggr) + 1 , \numberthis \label{equ:c3}
\end{align*}
for $E_1(t) = \int_t^\infty \frac{e^{-z}}{z}dz$ being the exponential integral \cite{milton1964handbook}.
\end{lemma}
\noindent The proof of Lemma~\ref{lem:scalar} is given in Appendix~\ref{appdx:scalar}. 

\noindent \textbf{Step 4:}
Finally, we show the convergence of $\unorm{\bm \delta^{(k)}}$ based on the convergence of $\unorm{\tilde{\bm \delta}^{(k)}}$.
From (\ref{equ:k_tilde}), substituting $\tilde{\epsilon} = \epsilon/\kappa(\bm Q)$ and identifying $c_3$ as $c_3(\rho(\bm H), \tau)$, we obtain (\ref{equ:k}).
Thus, it remains to prove that the accuracy on the transformed error vector $\unorm{\tilde{\bm \delta}^{(k)}} \leq \tilde{\epsilon} \unorm{\tilde{\bm \delta}^{(0)}}$ is sufficient for the accuracy on the original error vector $\unorm{{\bm \delta}^{(k)}} \leq \epsilon \unorm{{\bm \delta}^{(0)}}$.
Indeed, given
\begin{align*}
    \unorm{\tilde{\bm \delta}^{(k)}} \leq \tilde{\epsilon} \unorm{\tilde{\bm \delta}^{(0)}} = \frac{\epsilon}{\unorm{\bm Q}_2 \unorm{\bm Q^{-1}}_2} \unorm{\tilde{\bm \delta}^{(0)}} ,
\end{align*}
we have
\begin{align*}
    \unorm{\bm \delta^{(k)}} = \unorm{\bm Q \tilde{\bm \delta}^{(k)}} &\leq \unorm{\bm Q}_2 \unorm{\tilde{\bm \delta}^{(k)}} \\
    &\leq \unorm{\bm Q}_2 \frac{\epsilon}{\unorm{\bm Q}_2 \unorm{\bm Q^{-1}}_2} \unorm{\tilde{\bm \delta}^{(0)}} \\
    &= \frac{\epsilon}{\unorm{\bm Q^{-1}}_2} \unorm{\tilde{\bm \delta}^{(0)}} \leq \epsilon \unorm{\bm \delta^{(0)}} ,
\end{align*}
where the last inequality stems from $\unorm{\tilde{\bm \delta}^{(0)}} = \unorm{\bm Q^{-1} {\bm \delta}^{(0)}} \leq \unorm{\bm Q^{-1}}_2 \unorm{\bm \delta^{(0)}}$.
This completes our proof of Theorem~\ref{theo:pgd}.

\section{Applications}
\label{sec:apps}

In this section, we demonstrate the application of our proposed framework to a collection of well-known problems in machine learning and signal processing. 
The constraint sets in these problems vary from as simple as an affine subspace (A) and a sphere (C) to more complex algebraic varieties such as the $s$-sparse vector set (B) and the low-rank matrix set (D). 
We consider both problems with known convergence rate results and problems for which the rate is unavailable. The former allows us to verify the correctness of our analysis against the known rate results, while for the latter numerical experiments are used to verify the rate. Additionally, we illustrate how ROC can be obtained for each problem.
Due to space limitation, we restrict the illustration of our framework to the four aforementioned applications. 
While we believe that additional applications can be considered (see the potential applications of our framework in  Section~V), such applications may require a more elaborate development.
Our goal in this section is to offer a recipe for analyzing the convergence of PGD for different applications using the proposed framework.
Table~\ref{tbl:recipe} describes the steps we follow to obtain the asymptotic linear rate and the region of linear convergence in each application.
Table~\ref{tbl:result} summarizes our local convergence results on the four problems presented in this section.
The detailed analysis is given below.

\subsection{Linear Equality-Constrained Least Squares}\label{subsec:LLCS}

As a sanity check, we start with a simple example of the so-called linear equality-constrained least squares (LECLS)
\begin{align} \label{prob:CLS}
    \boxed{\min_{\bm x \in \R^n}\frac{1}{2} \unorm{\bm A \bm x - \bm b}^2 \quad \text{s.t. } \bm C \bm x = \bm d ,}
\end{align}
where $\bm A \in \R^{m \times n}$, $\bm b \in \R^m$, $\bm C \in \R^{p \times n}$, and $\bm d \in \R^p$. 
In addition, we assume that $p < n$ and $\bm C$ has linearly independent rows.
The LECLS problem finds application in a wide range of areas such as linear-phase system identification \cite{hong2013matrix}, antenna array processing \cite{de2004constrained}, and adaptive array processing \cite{frost1972algorithm}.
While this problem can be solved efficiently using the method of Lagrange multipliers \cite{resende1996fast} or the method of weighting \cite{van1985method}, we limit our interest to using PGD to solve (\ref{prob:CLS}) to demonstrate the applicability of our analysis.
In the literature, this algorithm is referred to as the projected Landweber iteration \cite{dunn1981global,bertsekas1982projection,luo1993error,johansson2006application}. While these works provide bounds on the linear convergence of PGD for different variants of linear equality-constrained problems, we have not found any closed-form expression of the asymptotic rate of linear convergence.

\begin{table}[t]
    \centering
    \begin{tabularx}{\columnwidth}{|l X|}
        \hline
        \textbf{Step 1:} & Identify $\bm A$, $\bm b$, $\C$, and $\P_\C$.  \\
        \textbf{Step 2:} & Establish the conditions for $\bm x^* \in \C$ to be a \textbf{Lipschitz stationary point} of (\ref{prob:f}). In particular, {\em(i)} $\P_\C$ is Lipschitz-continuously differentiable at every $\bm x^*$ with $\nabla \P_{\C} (\bm x^*)$, $c_1(\bm x^*)$, and $c_2(\bm x^*)$; and {\em(ii)} the stationarity equation (\ref{equ:stationary}) holds. \\
        \textbf{Step 3:} & Establish the conditions for $\eta>0$ such that {\em(i)} $\bm x^*$ is a \textbf{fixed point} of Algorithm~\ref{algo:PGD} with step size $\eta$, i.e., $\bm x^* = \P_\C(\bm z^*_\eta)$, for $\bm z^*_\eta = \bm x^* - \eta \bm A^\topnew ( \bm A \bm x^*  - \bm b )$; and {\em(ii)} $\P_\C$ is Lipschitz-continuously differentiable at $\bm z^*_\eta$ with $\nabla \P_\C(\bm z^*_\eta)$, $c_1(\bm z^*_\eta)$, and $c_2(\bm z^*_\eta)$. \\
        \textbf{Step 4:} & Determine the \textbf{asymptotic linear rate} $\rho$ as the spectral radius of $\bm H$ given by (\ref{equ:H}). (If $\nabla \P_{\C} (\bm z^*_\eta) = \nabla \P_\C (\bm x^*)$, (\ref{equ:rho_noiseless}) can be used instead.) \\
        \textbf{Step 5:} & Establish the conditions for $\rho<1$, which guarantees local linear convergence. Thereby, combine these conditions with the previous conditions obtained from Steps~2 and 3. \\
        \textbf{Step 6:} & If $\bm H$ is diagonalizable, determine the \textbf{region of linear convergence} given by (\ref{cond:roc}). (If $\nabla \P_{\C} (\bm z^*_\eta) = \nabla \P_\C (\bm x^*)$, (\ref{cond:roc_sym}) can be used instead.) \\
        \hline
    \end{tabularx}
    \caption{General recipe for local convergence analysis.}
    \label{tbl:recipe}
\end{table}

\renewcommand{\arraystretch}{3}
\begin{table*}
\begin{center}
\resizebox{\linewidth}{!}{
\begin{tabular}{|l|l|l|l|}
\hline
\textbf{Problem formulation} & \textbf{Condition(s) for linear convergence} & \textbf{Asymptotic rate of convergence $\rho$} & \textbf{Region of convergence} \\
\hline
\makecell{$\min \frac{1}{2} \unorm{\bm A \bm x - \bm b}^2$ \\ s.t. $\bm C \bm x = \bm d$} & $\begin{cases} \bm K = (\bm A \bm V_C^\perp)^{\topnew} \bm A \bm V_C^\perp \text{ has full rank } \\ 0 < \eta < 2/\unorm{\bm A \bm V_C^\perp}_2^2 \end{cases}$ & $\max \{ \uabs{1-\eta \lambda_1(\bm K)} , \uabs{1-\eta \lambda_{n-p}(\bm K)} \}$ & $\unorm{\bm x - \bm x^*} < \infty$ \\
\hline 
\makecell{$\min \frac{1}{2} \unorm{\bm A \bm x - \bm b}^2$ \\ s.t. $\unorm{\bm x}_0 \leq s$} & $\begin{cases} \bm K = (\bm A \bm S_{\bm x^*})^{\topnew} \bm A \bm S_{\bm x^*} \text{ has full rank } \\ 0 < \eta < \min \{ \frac{2}{\unorm{\bm A \bm S_{\bm x^*}}_2^2} , \frac{\uabs{x^*_{[s]}}}{\unorm{\bm v^*}_\infty} \} \end{cases}$ & $\max \{ \abs{1-\eta \lambda_1(\bm K)} , \abs{1-\eta \lambda_s(\bm K)} \}$ & $\unorm{\bm x - \bm x^*} < \min \{ \frac{\uabs{x^*_{[s]}}}{\sqrt{2}}, \frac{\uabs{x^*_{[s]}} - \eta \unorm{\bm v^*}_\infty}{\sqrt{2} \unorm{\bm I_n - \eta \bm A^\topnew \bm A}_2} \}$ \\
\hline 
\makecell{$\min \frac{1}{2} \unorm{\bm A \bm x - \bm b}^2$ \\ s.t. $\unorm{\bm x} = 1$} & $\begin{cases} 0<\eta<\infty \quad \text{ if } \gamma \leq -\lambda_1(\bm K) \\ 0<\eta<\frac{2}{\gamma + \lambda_1(\bm K)} \quad \text{ o.t.w.} \end{cases}$ & $\frac{1}{{1-\eta \gamma}} \max \{ \uabs{1 - \eta \lambda_1(\bm K)} , \uabs{1 - \eta \lambda_{n-1}(\bm K)} \} $ & $\unorm{\bm x - \bm x^*} \leq \frac{1-\rho}{2(t^2 + t)}, ~t = \frac{\unorm{\bm I_n - \eta \bm A^\topnew \bm A}_2} {1-\eta \gamma}$ \\
\hline
\makecell{$\min \frac{1}{2} \unorm{\P_{\Omega} (\bm X - \bm M)}_F^2$ \\ s.t. $\rank(\bm X) \leq r$} & $\begin{cases} \bm K = \bm Q_{\perp}^\topnew \bm S_\Omega \bm S_{\Omega}^\topnew \bm Q_{\perp} \text{ has full rank} \\ 0 < \eta < \frac{2}{\unorm{\bm Q_{\perp}^\topnew \bm S_\Omega \bm S_{\Omega}^\topnew \bm Q_{\perp}}_2} \end{cases}$ & $\max \{ \uabs{1 - \eta \lambda_1(\bm K)} , \uabs{1 - \eta \lambda_{r(m+n-r)}(\bm K)} \}$ & $\unorm{\bm X - \bm X^*}_F \leq \frac{(1 - \rho) \sigma_r(\bm X^*)}{8(1+\sqrt{2})}$ \\
\hline 
\end{tabular}
}
\caption{Summary of local convergence analysis for four problems: linear equality-constrained least squares (Sec.~\ref{subsec:LLCS}), sparse recovery (Sec.~\ref{subsec:IHTSR}), least squares with a unit norm constraint (Sec.~\ref{subsec:LSUNC}), and matrix completion (Sec.~\ref{subsec:MCP}). 
In the second row, $\bm v^* = \bm A^\topnew (\bm A \bm x^* - \bm b)$.
In the third row, $\bm K = (\bm A \bm U_{\bm x^*})^{\topnew} \bm A \bm U_{\bm x^*}$.
We refer the reader to each of the corresponding sections for further details.}
\label{tbl:result}
\end{center}
\end{table*}

\noindent \textbf{Step 1:} 
In this example, $\bm A$ and $\bm b$ are given explicitly in (\ref{prob:CLS}). The constraint set $\C$ is the closed convex affine subspace 
\begin{align*}
    \C = \{ \bm x \in \R^n \mid \bm C \bm x = \bm d \} .
\end{align*}
The orthogonal projection onto this subspace is given in a closed-form expression as $\P_{\C} (\bm x) = \bm x - \bm C^{\topnew} (\bm C \bm C^{\topnew})^{-1} (\bm C \bm x - \bm d)$, for all $\bm x \in \R^n$ \cite{meyer2000matrix}.
Since $\bm C$ has full row rank, it admits a compact singular value decomposition (SVD) $\bm C = \bm U_C \bm \Sigma_C \bm V_C^\topnew$, where $\bm \Sigma_C \in \R^{p \times p}$ is a diagonal matrix with positive diagonal entries, $\bm U_C \in \R^{p \times p}$ and $\bm V_C \in \R^{n \times p}$ satisfy $\bm U_C^\topnew \bm U_C = \bm V_C^\topnew \bm V_C = \bm I_p$. 
Denote $\bm V_C^\perp \in \R^{n \times (n-p)}$ the orthogonal complement of $\bm V_C$, i.e., $\bm V_C^\perp (\bm V_C^\perp)^\topnew = \bm I_n - \bm V_C \bm V_C^\topnew$ and $(\bm V_C^\perp)^\topnew \bm V_C^\perp = \bm I_{n-p}$.
Substituting the SVD of $\bm C$ back into the aforementioned expression of $\P_\C$ yields
\begin{align*}
    \P_\C (\bm x) = \bm V_C^\perp (\bm V_C^\perp)^\topnew \bm x + \tilde{\bm d} , \numberthis \label{equ:CLS_P0}
\end{align*}
where $\tilde{\bm d} = \bm V_C \bm \Sigma_C^{-1} \bm U_C^\topnew \bm d = \bm C^\dagger \bm d$.

\noindent \textbf{Step 2:} 
From (\ref{equ:CLS_P0}), we obtain the difference between the two projections of $\bm x+\bm \delta$ and $\bm x$ onto $\C$, for any $\bm x, \bm \delta \in \R^n$, as $\P_\C(\bm x + \bm \delta) - \P_\C(\bm x) = \bm V_C^\perp (\bm V_C^\perp)^\topnew \bm \delta$.
Using Definition~\ref{def:Pc_diff} with the note that $\Pi_\C(\bm x + \bm \delta)$ is always singleton, we have $\P_\C$ is Lipschitz-continuously differentiable at every $\bm x \in \R$ with 
\begin{align} \label{equ:PT_Lipschitz}
    \nabla \P_\C(\bm x) = \bm V_C^\perp (\bm V_C^\perp)^\topnew ,~ c_1(\bm x)=\infty ,~ c_2(\bm x)=0 .
\end{align}
Due to the independence from $\bm x$, we also have $\P_\C$ is Lipschitz-continuously differentiable at every $\bm x^* \in \C$ with 
\begin{align*}
    \nabla \P_\C(\bm x^*) = \bm V_C^\perp (\bm V_C^\perp)^\topnew ,~ c_1(\bm x^*)=\infty ,~ c_2(\bm x^*)=0 .
\end{align*}
Next, substituting $\nabla \P_\C(\bm x^*) = \bm V_C^\perp (\bm V_C^\perp)^\topnew$ into the stationarity equation (\ref{equ:stationary}) yields
$\bm V_C^\perp (\bm V_C^\perp)^\topnew \bm A^\topnew\bigl( \bm A \bm x^* - \bm b \bigr) = \bm 0$.
Since $\bm V_C^\perp \in \R^{n \times (n-p)}$ has full-rank, we can omit the left most $\bm V_C^\perp$ and obtain the condition for $\bm x^* \in \C$ to be a Lipschitz stationary point of (\ref{prob:CLS}) as
\begin{align} \label{equ:LCLS_AV}
    (\bm A \bm V_C^\perp)^\topnew ( \bm A \bm x^* - \bm b ) = \bm 0 ,
\end{align}
which means $\bm A \bm x^* - \bm b$ is in the left null space of $\bm A \bm V_C^\perp$.\footnote{Here, it is interesting to note that any stationary point of (\ref{prob:CLS}) is a global minimizer since (\ref{prob:CLS}) is a convex optimization problem.}

\noindent \textbf{Step 3:} 
Evaluating the projection in (\ref{equ:CLS_P0}) at $\bm z_\eta^* = \bm x^* - \eta \bm A^\topnew ( \bm A \bm x^*  - \bm b )$ and using the stationarity condition (\ref{equ:LCLS_AV}) to eliminate the term $\eta \bm V_C^\perp (\bm V_C^\perp)^\topnew \bm A^\topnew ( \bm A \bm x^*  - \bm b )$, we have $ \P_\C (\bm z_\eta^*) = \bm x^*$ for any $\eta>0$.
Thus, the condition in this step for $\bm x^*$ to be a fixed point of Algorithm~\ref{algo:PGD} is $\eta>0$.
In addition, substituting $\bm x = \bm z_\eta^*$ into (\ref{equ:PT_Lipschitz}), we obtain $\P_\C$ is Lipschitz-continuously differentiable at $\bm z^*_\eta$ with
\begin{align*}
    \nabla \P_\C(\bm z^*_\eta) = \bm V_C^\perp (\bm V_C^\perp)^\topnew ,~ c_1(\bm z^*_\eta)=\infty ,~ c_2(\bm z^*_\eta)=0 .
\end{align*}

\noindent \textbf{Step 4:} 
Since $\nabla \P_{\C} (\bm z^*_\eta) = \nabla \P_\C (\bm x^*) = \bm V_C^\perp (\bm V_C^\perp)^\topnew$, using (\ref{equ:rho_noiseless}), we obtain the asymptotic linear rate as
\begin{align*} 
    \rho &= \max \{ \uabs{1-\eta \lambda_1} , \uabs{1-\eta \lambda_{n-p}} \} , \numberthis \label{equ:rho_egC}
\end{align*}
where $\lambda_1$ and $\lambda_{n-p}$ are the largest and smallest eigenvalues of $(\bm A \bm V_C^\perp)^{\topnew} \bm A \bm V_C^\perp$, respectively.

\noindent \textbf{Step 5:}
From (\ref{equ:rho_egC}), we have $\rho < 1$ if and only if $(\bm A \bm V_C^\perp)^{\topnew} \bm A \bm V_C^\perp$ has full rank and $0 < \eta < 2/\unorm{\bm A \bm V_C^\perp}_2^2$. It is noted that the latter condition is sufficient for the condition $\eta>0$ in Step~3.

\noindent \textbf{Step 6:} Since $c_1(\bm x^*) = c_1(\bm z^*_\eta) = \infty$ and $c_2(\bm x^*) = c_2(\bm z^*_\eta) = 0$, the region of convergence given by (\ref{cond:roc_sym}) is the entire space $\R^n$, which implies global convergence.

\begin{remark}
The explicit expression of the convergence rate in (\ref{equ:rho_egC}) offers a simple method to select the optimal step size:
\begin{align*}
    \eta_{opt} &= \argmin_{0<\eta<2/\unorm{\bm A \bm V_C^\perp}_2^2} \max \{ \uabs{1-\eta \lambda_1} , \uabs{1-\eta \lambda_{n-p}} \} \\
    &= \frac{2}{\lambda_{1}+\lambda_{n-p}} . \numberthis \label{equ:CLS_eta_opt}
\end{align*}
Using $\eta = \eta_{opt}$, we obtain the optimal rate of convergence
\begin{align*}
    \rho_{opt} &= 1 - \frac{2}{\kappa((\bm A \bm V_C^\perp)^{\topnew} \bm A \bm V_C^\perp) + 1} . \numberthis \label{equ:CLS_rho_opt}
\end{align*}
As a comparison, the optimal convergence rate of gradient descent for the unconstrained problem is given by \cite{polyak1987introduction}
\begin{align*}
    u_{opt} = 1 - \frac{2}{\kappa(\bm A^\topnew \bm A) + 1} .
\end{align*}
Recall from Remark~\ref{rmk:interlace} that $\rho_{opt} \leq u_{opt}$ due to the interlacing of eigenvalues of $(\bm A \bm V_C^\perp)^{\topnew} \bm A \bm V_C^\perp$ and $\bm A^\topnew \bm A$.

\end{remark}

\subsection{Iterative Hard Thresholding for Sparse Recovery}\label{subsec:IHTSR}

In compressed sensing, one would like to reconstruct a sparse signal by finding solutions to under-determined linear systems $\bm A \bm x = \bm b$, where $\bm A \in \R^{m \times n}$ and $\bm b \in \R^m$ (for $m < n$). This problem can be formulated as an L0-norm constrained least squares:
\begin{align} \label{prob:sparse}
    \boxed{\min_{\bm x \in \R^n} \frac{1}{2} \unorm{\bm A \bm x - \bm b}^2 \quad \text{s.t. } \unorm{\bm x}_0 \leq s .}
\end{align}
In the literature, the PGD algorithm for solving (\ref{prob:sparse}) is often known as iterative hard thresholding (IHT), with myriad applications in medical imaging \cite{dogandvzic2011mask}, MIMO communication \cite{gao2014priori,stockle2016channel}, antenna arrays \cite{stoeckle2015doa}, and scene recognition \cite{yu2018homotopy}.
The convergence of a special case of IHT in which $\unorm{\bm A}_2 < 1$ and $\eta=1$ has been well-studied in \cite{blumensath2008iterative, blumensath2009iterative}, under the restricted isometry property (RIP) assumption on $\bm A$.
In the following, we demonstrate the application of our framework to establishing a local convergence analysis of IHT with a range of different step sizes, without requiring the RIP of $\bm A$.

\noindent \textbf{Step 1:} 
In this example, $\bm A$ and $\bm b$ are given explicitly in (\ref{prob:sparse}), and the constraint set $\C$ is the closed non-convex set of $s$-sparse vectors
\begin{align*}
    \C = \{ \bm x \in \R^n \mid \unorm{\bm x}_0 \leq s \} ,
\end{align*}
with the projection $\P_\C: \R^n \to \R^n$ given by \cite{blumensath2008iterative}
\begin{align} \label{equ:sparse_P}
    [\P_\C (\bm x)]_i = \begin{cases}
    0 & \text{if } \uabs{x_i} < \uabs{x_{[s]}}  \\
    x_i & \text{if } \uabs{x_i} \geq \uabs{x_{[s]}} 
    \end{cases} \text{ for } i=1,\ldots,n ,
\end{align}
where $x_i$ and $x_{[s]}$ denote the $i$th coordinate and the $s$th largest (in magnitude) element of a vector $\bm x \in \R^n$, respectively. In the case $\bm x$ has multiple elements with the same magnitude as $x_{[s]}$, e.g., $x_{[s]} = x_{[s+1]} > 0$, we sort these entries based on the (descending) lexicographical order so that (\ref{equ:sparse_P}) is well-defined (see \cite{blumensath2008iterative}-p.~10).

\noindent \textbf{Step 2:}
In contrast to the previous example, the projection here is nonlinear and non-unique since the set $\C$ is a real algebraic variety but not smooth in those points in $\R^n$ of sparsity strictly less than $s$. 
The smooth part of $\C$ is the subset
\begin{align*}
    \C_{=s} = \{ \bm x \in \R^n \mid \unorm{\bm x}_0 = s \}
\end{align*}
of vectors with exactly $s$ non-zero elements. 
 
In Supplementary Material Section~III, we show that any $\bm x^* \in \Phi_{= s}$ and $\bm x \in \B(\bm x^*, \uabs{x^*_{[s]}}/\sqrt{2})$ share the same index set of $s$-largest elements (in magnitude), denoted by $\Omega_s(\bm x^*)$.\footnote{It is interesting to note that $\uabs{x^*_{[s]}}/\sqrt{2}$ is the largest possible radius. A counter-example is also constructed in Supplementary Material Section~III.}
Let the indices in $\Omega_s(\bm x^*)$ be $i_1 \leq \ldots \leq i_s$ and $\bm S_{\bm x^*} = [\bm e_{i_1}, \ldots , \bm e_{i_s}] \in \R^{n \times s}$.
Then, we have $(\bm S_{\bm x^*})^\topnew \bm S_{\bm x^*} = \bm I_s$ and
\begin{align} \label{equ:IHT_sqrt2}
    \P_\C (\bm x) = \bm S_{\bm x^*} \bm S_{\bm x^*}^{\topnew} \bm x , \qquad \forall \bm x \in \B(\bm x^*, \uabs{x^*_{[s]}}/\sqrt{2}) .
\end{align}
By Definition~\ref{def:Pc_diff}, we obtain $\P_\C$ is Lipschitz-continuously differentiable at any $\bm x^* \in \Phi_{= s}$ with
\begin{align*}
    \nabla \P_\C (\bm x^*) = \bm S_{\bm x^*} \bm S_{\bm x^*}^{\topnew} ,~ c_1(\bm x^*) = \frac{1}{\sqrt{2}} \uabs{x^*_{[s]}} ,~ c_2(\bm x^*) = 0 .
\end{align*}
Similar to the previous example, the stationarity equation for $\bm x^* \in \C_{=s}$ is given by
\begin{align} \label{equ:sparse_stationary}
    (\bm A \bm S_{\bm x^*})^\topnew ( \bm A \bm x^* - \bm b ) = \bm 0 .
\end{align}
Thus, we obtain the conditions for $\bm x^* \in \C$ to be a Lipschitz stationary point of (\ref{prob:sparse}) are $\bm x^* \in \C_{=s}$ and the vector $\bm v^* = \bm A^\topnew (\bm A \bm x^* - \bm b)$ satisfies $v^*_i=0$ for all $i \in \Omega_s(\bm x^*)$.

\noindent \textbf{Step 3:} 
First, following a similar approach to that in \cite{blumensath2008iterative}, we show that the condition in this step for $\bm x^*$ to be a fixed point of Algorithm~\ref{algo:PGD} is
\begin{align*}
    0<\eta< \frac{\uabs{x^*_{[s]}}}{\unorm{\bm v^*}_\infty} . \numberthis \label{equ:sparse_eta}
\end{align*}
Since $v^*_i=0$ for all $i \in \Omega_s(\bm x^*)$, we have $\bm z^*_\eta = \bm x^* - \eta \bm v^*$ satisfies $(\bm z^*_\eta)_i = \bm x^*_i$ for all $i \in \Omega_s(\bm x^*)$. Moreover, for any indices $i \in \Omega_s(\bm x^*)$ and $j \in \{1,\ldots,n\} \setminus \Omega_s(\bm x^*)$, we have
\begin{align*}
    \uabs{(z^*_\eta)_j} &= \uabs{x^*_j - \eta v^*_j} = \eta \uabs{v^*_j} \\
    &< \frac{\uabs{x^*_{[s]}}}{\unorm{\bm v^*}_\infty} \uabs{v^*_j} \leq \uabs{x^*_{[s]}} \leq \uabs{x^*_i} = \uabs{(z^*_\eta)_i} ,
\end{align*}
where the second inequality stems from $\uabs{v^*_j} \leq \unorm{\bm v^*}_\infty$.
Therefore, $\Omega_s(\bm x^*)$ contains the $s$-largest (in magnitude) elements of $\bm z^*_\eta$, and hence, $\bm x^* = \P_\C(\bm z^*_\eta)$.

Second, we consider the Lipschitz-continuous differentiability of $\P_\C$ at $\bm z^*_\eta$.
Given $\eta$ in (\ref{equ:sparse_eta}), by the same argument as in Supplementary Material Section~III, one can show that every point in $\B(\bm z^*_\eta, (\uabs{(z^*_\eta)_{[s]}} - \uabs{(z^*_\eta)_{[s+1]}})/\sqrt{2})$ shares the same index set of $s$-largest elements (in magnitude) with $\bm z^*_\eta$, which is $\Omega_s(\bm x^*)$.
Here, we note that $\uabs{(z^*_\eta)_{[s]}} - \uabs{(z^*_\eta)_{[s+1]}} = \uabs{x^*_{[s]}} - \eta \unorm{\bm v^*}_\infty$.
Thus, we obtain $\P_\C$ is Lipschitz-continuously differentiable at $\bm z^*_\eta$ with 
\begin{align*}
    &\nabla \P_\C (\bm z^*_\eta) = \bm S_{\bm x^*} \bm S_{\bm x^*}^{\topnew} , \\
    &c_1(\bm z^*_\eta) = \frac{1}{\sqrt{2}} \bigl( \uabs{x^*_{[s]}} - \eta \unorm{\bm v^*}_\infty \bigr) , \quad c_2(\bm z^*_\eta) = 0 .
\end{align*}

\noindent \textbf{Step 4:} 
Since $\nabla \P_{\C} (\bm z^*_\eta) = \nabla \P_\C (\bm x^*) = \bm S_{\bm x^*} \bm S_{\bm x^*}^{\topnew}$, using (\ref{equ:rho_noiseless}), we obtain the asymptotic linear rate as
\begin{align*} 
    \rho &= \max \{ \abs{1-\eta \lambda_1} , \abs{1-\eta \lambda_s} \} . \numberthis \label{equ:rho_sparse}
\end{align*}
where $\lambda_1$ and $\lambda_s$ are the largest and smallest eigenvalues of $(\bm A \bm S_{\bm x^*})^{\topnew} \bm A \bm S_{\bm x^*}$, respectively.

\noindent \textbf{Step 5:} 
From (\ref{equ:rho_sparse}), $\rho < 1$ if and only if $(\bm A \bm S_{\bm x^*})^{\topnew} \bm A \bm S_{\bm x^*}$ has full rank and
\begin{align} \label{equ:sparse_eta2}
    0 < \eta < \frac{2}{\unorm{\bm A \bm S_{\bm x^*}}_2^2} .
\end{align}
Combining (\ref{equ:sparse_eta}) and (\ref{equ:sparse_eta2}) yields the condition on the step size
\begin{align} \label{equ:sparse_eta_roc}
    0 < \eta < \min \biggl\{ \frac{2}{\unorm{\bm A \bm S_{\bm x^*}}_2^2} , \frac{\uabs{x^*_{[s]}}}{\unorm{\bm v^*}_\infty} \biggr\} .
\end{align}
Here, we note that the condition $(\bm A \bm S_{\bm x^*})^{\topnew} \bm A \bm S_{\bm x^*}$ has full rank is related to the restricted isometry property (RIP) assumption on $\bm A$:
$(1-\delta_s) \norm{\bm x}^2 \leq \norm{\bm A \bm x}^2 \leq (1+\delta_s) \norm{\bm x}^2$, for $\delta_s \in (0,1)$ and any $s$-sparse vector $\bm x \in \R^n$ \cite{candes2005decoding}. In the reduced-form, we can rewrite the RIP assumption as
\begin{align} \label{equ:RIP}
    0 < (1-\delta_s) \norm{\bm y}^2 \leq \norm{\bm A \bm S \bm y}^2 \leq (1+\delta_s) \norm{\bm y}^2 ,
\end{align}
for any $\bm y \in \R^s$ and any selection matrix $\bm S \in \R^{n \times s}$ obtained by randomly choosing $s$ columns from the $n \times n$ identity matrix. Substituting $\bm S = \bm S_{\bm x^*}$ into (\ref{equ:RIP}), we obtain $(\bm A \bm S_{\bm x^*})^{\topnew} \bm A \bm S_{\bm x^*}$ has full rank.

\noindent \textbf{Step 6:}
Recall that $c_2(\bm x^*) = c_2(\bm z^*_\eta) = 0$.
From (\ref{cond:roc_sym}), the region of convergence is given by
\begin{align} \label{equ:sparse_roc}
    \unorm{\bm x - \bm x^*} < \min \biggl\{ \frac{\uabs{x^*_{[s]}}}{\sqrt{2}}, \frac{\uabs{x^*_{[s]}} - \eta \unorm{\bm v^*}_\infty}{\sqrt{2} \unorm{\bm I_n - \eta \bm A^\topnew \bm A}_2} \biggr\} . 
\end{align}

\begin{figure}[t]
    \centering
    \includegraphics[scale=.61]{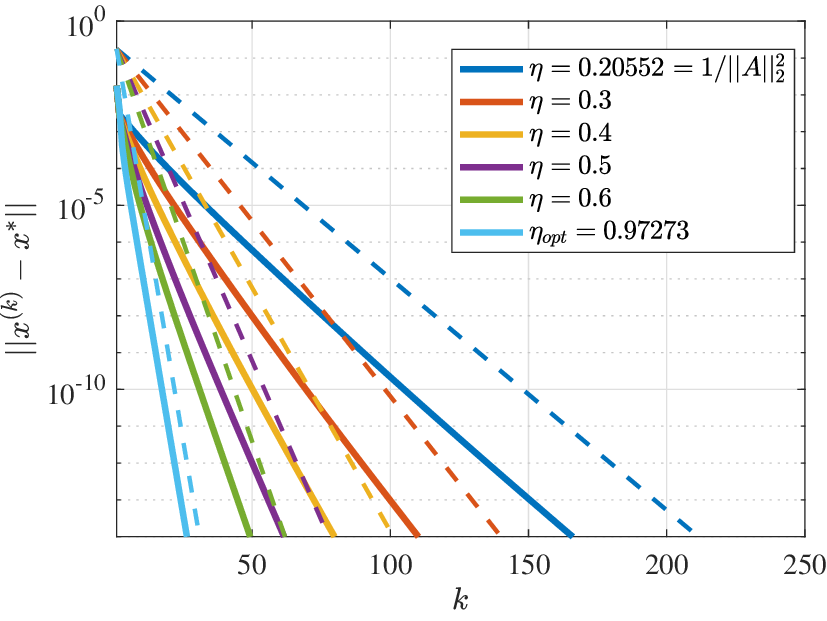}
    \caption{(Log-scale) plot of the distance between the current iterate and the local minimizer of the sparse recovery problem, as a function of the number of iterations. Each solid line corresponds to PGD with a different fixed step size. Each dashed line represents the respective exponential bound $\rho^k$ up to a constant, where the theoretical rate $\rho$ is given by (\ref{equ:rho_sparse}). In the experiment, we select $m=200, n=300$, and $s=10$. The optimal step size $\eta_{opt}=0.97273$ is computed by (\ref{equ:sparse_opt}), with the corresponding optimal rate $\rho_{opt}=0.3613$.}
    \label{fig:sparse}
\end{figure}

\begin{remark}
Similar to (\ref{equ:CLS_eta_opt}) and (\ref{equ:CLS_rho_opt}), the optimal step size and the optimal convergence rate are given by
\begin{align*}
    \eta_{opt} &= \frac{2}{\lambda_{1}((\bm A \bm S_{\bm x^*})^{\topnew} \bm A \bm S_{\bm x^*})+\lambda_{s}((\bm A \bm S_{\bm x^*})^{\topnew} \bm A \bm S_{\bm x^*})} , \\
    \rho_{opt} &= 1 - \frac{2}{\kappa((\bm A \bm S_{\bm x^*})^{\topnew} \bm A \bm S_{\bm x^*}) + 1} . \numberthis \label{equ:sparse_opt}
\end{align*}
We consider the following numerical experiment to verify the analytical rate in (\ref{equ:rho_sparse}). We start by generating $\bm A$, $\bm x^*$, and $\bm b$ as follows. First, we sample an $200 \times 300$ sensing matrix $\bm A$ with $i.i.d$ Gaussian distributed entries $\mathcal{N}(0,1/200)$.\footnote{Note that such random matrix is shown to satisfy the RIP constraint \cite{candes2005decoding}.}
Next, we create a $10$-sparse solution $\bm x^*$ by randomly selecting $10$ coordinates and assigning non-zero values to them based on $i.i.d$ normal distribution $\mathcal{N}(0,1)$.
Finally, we set $\bm b = \bm A \bm x^*$.
We apply PGD with different step sizes (listed in Fig.~\ref{fig:sparse}) including $\rho_{opt}$ in (\ref{equ:sparse_opt}) and record the value of $\norm{\bm x^{(k)} - \bm x^*}$ as a function of $k$. In Fig.~\ref{fig:sparse}, the aforementioned curves are presented along with their analytic bounds given by $\rho^k$ (up to a constant). 
The match in the slope between the analytic rate curve and the empirical rate curve verifies the analytic rate predicts accurately the asymptotic rate obtained empirically.
\end{remark}

\begin{remark} \label{rmk:IHT_local_min}
In Supplementary Material Section~III, we further show that any stationary point $\bm x^*$ must be a local minimum of (\ref{prob:sparse}). Moreover, the condition $(\bm A \bm S_{\bm x^*})^{\topnew} \bm A \bm S_{\bm x^*}$ has full rank in Step~5 implies $\bm x^*$ is a \textbf{strict} local minimum of (\ref{prob:sparse}). 
Finally, it is interesting to note that in \cite{blumensath2008iterative}, the authors assume $\norm{\bm A}_2<1$ and select $\eta=1$. With these assumptions, the rate in (\ref{equ:rho_sparse}) simplifies to $\rho = 1-\lambda_s((\bm A \bm S_{\bm x^*})^{\topnew} \bm A \bm S_{\bm x^*}) = \norm{\bm I_s - (\bm A \bm S_{\bm x^*})^{\topnew} \bm A \bm S_{\bm x^*}}_2$, which is consistent with Eqn.~(3.9) in \cite{blumensath2008iterative}.
\end{remark}

\subsection{Least Squares with the Unit Norm Constraint}\label{subsec:LSUNC}

A common constraint that arises in regularization methods for ill-posed problems is the spherical constraint \cite{tarantola2005inverse,menke2018geophysical,hager2001minimizing}. In particular, we consider the following optimization problem
\begin{align} \label{prob:sphere}
    \boxed{\min_{\bm x \in \R^n} \frac{1}{2} \unorm{\bm A \bm x - \bm b}^2 \quad \text{s.t. } \unorm{\bm x} = 1 ,}
\end{align}
where $\bm A \in \R^{m \times n}$ and $\bm b \in \R^m$. 

\noindent \textbf{Step 1:} 
In this example, $\bm A$ and $\bm b$ are given explicitly in (\ref{prob:sphere}), and the constraint set $\C$ is the closed non-convex sphere 
\begin{align*}
    \C = \{ \bm x \in \R^n \mid \unorm{\bm x}=1 \} ,
\end{align*}
with the projection $\P_\C: \R^n \to \R^n$ given by
\begin{align} \label{equ:sphere_P}
    \P_{\C} (\bm x) = \begin{cases}
        \frac{\bm x}{\unorm{\bm x}} &\text{if } \bm x \neq \bm 0 , \\
        \bm e_1 &\text{if } \bm x = \bm 0 .
    \end{cases}  
\end{align}

\noindent \textbf{Step 2:}
In Example~\ref{eg:c2_Ps}, we showed that the projection onto the unit sphere is Lipschitz-continuously differentiable at any $\bm x \neq \bm 0$. 
Since $\bm 0 \not \in \C$, we have $\P_\C$ is Lipschitz-continuously differentiable at every $\bm x^* \in \C$ with 
\begin{align*}
    \nabla P_{\C}(\bm x^*) = \bm I_n - \bm x^* (\bm x^*)^\topnew ,~c_1(\bm x^*)=\infty ,~c_2(\bm x^*)=2 .
\end{align*}
In addition, substituting $\nabla P_{\C}(\bm x^*) = \bm I_n - \bm x^* (\bm x^*)^\topnew$ into the stationarity equation (\ref{equ:stationary}) yields $\bigl( \bm I_n - \bm x^* (\bm x^*)^\topnew \bigr) \bm A^\topnew (\bm A \bm x^* - \bm b) = \bm 0$.
Equivalently, we have 
\begin{align} \label{equ:sphere_stationary}
    \bm A^\topnew (\bm A \bm x^* - \bm b) = \gamma \bm x^* ,
\end{align}
where $\gamma = (\bm x^*)^\topnew \bm A^\topnew (\bm A \bm x^* - \bm b)$ is the Lagrange multiplier at $\bm x^*$ (see Lemma~1 in \cite{vu2019convergence}). 
Thus, we obtain the condition for $\bm x^* \in \C$ to be a Lipschitz stationary point of (\ref{prob:sphere}) is $\bm x^*$ and $\bm A^\topnew (\bm A \bm x^* - \bm b)$ are collinear.

\noindent \textbf{Step 3:}
First, the necessary condition for $\P_\C$ to be Lipschitz-continuously differentiable at $\bm z^*_\eta$ is $\bm z^*_\eta \in \sing \Pi_\C$, i.e., $\bm z^*_\eta \neq \bm 0$. From (\ref{equ:sphere_stationary}), we have $\bm z^*_\eta = \bm x^* - \eta \bm A^\topnew ( \bm A \bm x^* - \bm b ) = (1-\eta \gamma) \bm x^*$.
Hence, $\bm z^*_\eta \neq \bm 0$ is equivalent to $1-\eta \gamma \neq 0$.
Now the projection $\P_\C$ at $\bm z^*_\eta \neq \bm 0$ is given by
\begin{align*}
    \P_\C(\bm z^*_\eta) = \frac{(1-\eta \gamma) \bm x^*}{\unorm{(1-\eta \gamma) \bm x^*}} = \frac{1-\eta \gamma}{\uabs{1-\eta \gamma}} \bm x^* ,
\end{align*}
which implies $\bm x^* = \P_\C(\bm z^*_\eta)$ if and only if $1-\eta \gamma>0$.
Thus, we obtain the condition for $\eta>0$ such that $\bm x^*$ is a fixed point of Algorithm~\ref{algo:PGD} and $\P_\C$ is Lipschitz-continuously differentiable at $\bm z^*_\eta$ is $1-\eta \gamma>0$, which is equivalent to
\begin{align} \label{equ:sphere_eta0}
    \begin{cases}
        \eta \in (0,\infty) &\text{if } \gamma \leq 0 ,\\
        \eta \in (0, \frac{1}{\gamma}) &\text{if } \gamma > 0 .
    \end{cases}
\end{align}

Second, it follows from (\ref{equ:sphere_off}) that $\P_\C$ is Lipschitz-continuously differentiable at $\bm z^*_\eta$ with
\begin{align*}
    \nabla P_{\C}(\bm z^*_\eta) &= \frac{1}{\unorm{\bm z^*_\eta}} \biggl( \bm I_n - \frac{\bm z^*_\eta (\bm z^*_\eta)^\topnew}{\unorm{\bm z^*_\eta}^2} \biggr) = \frac{\bm I_n - \bm x^* (\bm x^*)^\topnew}{{1-\eta \gamma}} , \\
    c_1(\bm z^*_\eta) &= \infty , \quad c_2(\bm z^*_\eta) = \frac{2}{\unorm{\bm z^*_\eta}^2} = \frac{2}{(1-\eta \gamma)^2} .
\end{align*}

\noindent \textbf{Step 4:}
Denote $\bm P_{\bm x^*}^\perp = \bm I_n - \bm x^* (\bm x^*)^\topnew$. From (\ref{equ:H}), the asymptotic linear rate is given by
\begin{align*}
    \rho &= \rho \biggl( \frac{1}{{1-\eta \gamma}} \bm P_{\bm x^*}^\perp (\bm I_n - \eta \bm A^\topnew \bm A) \bm P_{\bm x^*}^\perp \biggr) \\
    &= \frac{1}{{1-\eta \gamma}} \unorm{\bm P_{\bm x^*}^\perp (\bm I_n - \eta \bm A^\topnew \bm A) \bm P_{\bm x^*}^\perp}_2 .
\end{align*}
Let $\bm P_{\bm x^*}^\perp = \bm U_{\bm x^*} \bm U_{\bm x^*}^\topnew$, where $\bm U_{\bm x^*} \in \R^{n \times (n-1)}$ is a semi-orthogonal matrix whose columns provide a basis for the null space of $\bm x^*$.
Then, following the same derivation as in the proof of Corollary~\ref{cor:noiseless}, we obtain
\begin{align} \label{equ:rho_sphere}
    \rho &= \frac{1}{{1-\eta \gamma}} \max \{ \uabs{1 - \eta \lambda_1} , \uabs{1 - \eta \lambda_{n-1}} \} ,
\end{align}
where $\lambda_1$ and $\lambda_{n-1}$ are the largest and smallest eigenvalues of $(\bm A \bm U_{\bm x^*})^{\topnew} \bm A \bm U_{\bm x^*}$, respectively.

\noindent \textbf{Step 5:}
Since $\abs{1-\eta \lambda}/(1-\eta \gamma) < 1$ is equivalent to $\eta \gamma -1 < 1-\eta \lambda < 1 -\eta \gamma$, we have $\rho < 1$ if and only if
\begin{align} \label{equ:sphere_local}
    \gamma < \lambda_{n-1}
\end{align}
and
\begin{align*}
     \eta \Bigl(\gamma + \lambda_1 \Bigr) < 2 . \numberthis \label{equ:eta_sphere2}
\end{align*}
Similar to (\ref{equ:sphere_eta0}), the inequality in (\ref{equ:eta_sphere2}) can be rewritten as
\begin{align*}
    \begin{cases}
        \eta \in (0,\infty) &\text{if } \gamma \leq -\lambda_1 ,\\
        \eta  \in (0, \frac{2}{\gamma + \lambda_1}) &\text{if } \gamma > -\lambda_1 .
    \end{cases}
\end{align*}
Finally, we note that conditions (\ref{equ:sphere_local}) and (\ref{equ:eta_sphere2}) together imply the condition $1-\eta \gamma > 0$ in Step~3 since $2 \eta \gamma < \eta (\gamma + \lambda_{n-1}) \leq  \eta (\gamma + \lambda_1) < 2$.

\noindent \textbf{Step 6:}
To determine the region of linear convergence, we first recall that $c_1(\bm x^*) = c_1(\bm z^*_\eta) = \infty$.
Second, we have
\begin{align*}
    \unorm{\nabla \P_\C (\bm z^*_\eta)}_2 = \norm{\frac{\bm I_n - \bm x^* (\bm x^*)^\topnew}{1-\eta \gamma}}_2 = \frac{1}{1-\eta \gamma} .
\end{align*}
Third, since $\bm H = \bm P_{\bm x^*}^\perp (\bm I_n - \eta \bm A^\topnew \bm A) \bm P_{\bm x^*}^\perp / (1-\eta \gamma)$ is symmetric, one can choose $\bm Q$ in the eigendecomposition $\bm H = \bm Q \bm \Lambda \bm Q^{-1}$ to be orthogonal, with $\kappa(\bm Q)=1$. 
Thus, from (\ref{cond:roc}), we obtain the region of linear convergence as
\begin{align*}
    \unorm{\bm x - \bm x^*} \leq \frac{1-\rho}{2(t^2 + t)} , \numberthis \label{equ:roc_sphere}
\end{align*}
where $t = \unorm{\bm I_n - \eta \bm A^\topnew \bm A}_2 / (1-\eta \gamma)$.

\begin{remark}
The local linear rate in (\ref{equ:rho_sphere}) matches the rate provided by Theorem~1 in \cite{vu2019convergence}. 
Compared to the setting in \cite{vu2019convergence}, here we consider a special case of the quadratic that is convex (and hence, $\lambda_d \geq 0$).
By minimizing the rate in (\ref{equ:rho_sphere}) over $\eta$, we also obtain the same optimal rate of linear convergence given by Lemma~5 in \cite{vu2019convergence}:
\begin{align*}
    \rho_{opt} = \frac{\lambda_1 - \lambda_{n-1}}{\lambda_1 + \lambda_{n-1} - 2\gamma} \text{ with } \eta_{opt} = \frac{2}{\lambda_1 + \lambda_{n-1}} .
\end{align*}
Interestingly, condition (\ref{equ:sphere_local}) implies $\bm x^*$ is a strict local minimum of (\ref{prob:sphere}) (see Lemma~2 in \cite{vu2019convergence}).
Since $\rho<1$ is one of the conditions in Theorem~\ref{theo:pgd}, our analysis requires $\bm x^*$ to be a strict local minimum of (\ref{prob:sphere}) in order to obtain linear convergence.
Finally, our framework provides the region of linear convergence in (\ref{equ:roc_sphere}), which is not given in \cite{vu2019convergence}.
\end{remark}

\subsection{Matrix Completion}\label{subsec:MCP}

\subsubsection{Background}

The last application is an application of our framework to the matrix case. In matrix completion \cite{candes2009exact}, given a rank-$r$ matrix $\bm M \in \R^{m \times n}$ (for $1 \leq r \leq \min\{m,n\}$) with a set of its observed entries indexed by $\Omega$, of cardinality $0<s<mn$, we wish to recover the unknown entries of $\bm M$ in the complement set $\bar{\Omega}$ by solving the following optimization:
\begin{align} \label{prob:mcp}
    \boxed{\min_{\bm X \in \R^{m \times n}} \frac{1}{2} \unorm{\P_{\Omega} (\bm X - \bm M)}_F^2 \text{ s.t. } \rank(\bm X) \leq r ,}
\end{align}
where $\P_{\Omega}: \R^{m \times n} \to \R^{m \times n}$ is the orthogonal projection onto the set of $m \times n$ matrices supported in $\Omega$, i.e.,
\begin{align*}
    [\P_\Omega (\bm X)]_{ij} = \begin{cases}
        X_{ij} &\text{ if } (i,j) \in \Omega , \\
        0 &\text{ if } (i,j) \not \in \Omega .
    \end{cases}
\end{align*}
It is noted that while $\bm M$ is unknown, the projection $\P_{\Omega}(\bm M)$ is unambiguously determined by the observed entries in $\bm M$. 
In the literature, the PGD algorithm for solving (\ref{prob:mcp}) is also known as the Singular Value Projection (SVP) algorithm for matrix completion \cite{jain2010guaranteed,chen2015fast,jain2015fast,ding2020leave}, with the update 
\begin{align*}
    \bm X^{(k+1)}= \P_{\mathcal{M}_{\leq r}} \bigl(\bm X^{(k)} - \eta \P_\Omega (\bm X^{(k)} - \bm M) \bigr) .
\end{align*}
Here, $\mathcal{M}_{\leq r}$ is the set of matrices of rank at most $r$, i.e.,
\begin{align*}
    \mathcal{M}_{\leq r} = \{ \bm X \in \R^{m \times n} \mid \rank(\bm X) \leq r \} .
\end{align*}
In addition, the orthogonal projection $\P_{\mathcal{M}_{\leq r}}: \R^{m \times n} \to \mathcal{M}_{\leq r}$ is defined by Eckart–Young–Mirsky theorem \cite{eckart1936approximation} as follows.
Let $\text{SVD}(\bm X)$ be the set of all triples $(\bm \Sigma, \bm U, \bm V)$ such that $\bm X = \bm U \bm \Sigma \bm V^\topnew$ and
\begin{align*}
    \begin{cases}
        \bm \Sigma = \diag(\sigma_1(\bm X),\ldots,\sigma_n(\bm X)) , \\
        \bm U \in \R^{m \times n}, \bm V \in \R^{n \times n}: \bm U^\topnew \bm U = \bm V^\topnew \bm V = \bm I_n .
    \end{cases}
\end{align*}
Denote $\bm u_i(\bm X)$ and $\bm v_i(\bm X)$ the $i$th columns of $\bm U$ and $\bm V$, respectively. Then, the set of all projections of $\bm X$ onto $\mathcal{M}_{\leq r}$ is given by
\begin{align*}
    \Pi_{\mathcal{M}_{\leq r}} (\bm X) = \Bigl\{ \sum_{i=1}^r &\sigma_i(\bm X) \bm u_i(\bm X) \bm v_i(\bm X)^\topnew \\
    &\mid (\bm \Sigma, \bm U, \bm V) \in \text{SVD}(\bm X) \Bigr\} . \numberthis  \label{equ:mcp_P}
\end{align*}
The set $\Pi_{\mathcal{M}_{\leq r}} (\bm X)$ is singleton if and only if $\sigma_r(\bm X)=0$ or $\sigma_r(\bm X) > \sigma_{r+1}(\bm X)$. 
In the case $\Pi_{\mathcal{M}_{\leq r}} (\bm X)$ has multiple elements, we define $\P_{\mathcal{M}_{\leq r}} (\bm X)$ as the greatest element in $\Pi_{\mathcal{M}_{\leq r}} (\bm X)$ based on the lexicographical order.
We re-emphasize that our subsequent analysis holds independently of this choice.

In differential geometry, it is well-known that $\mathcal{M}_{\leq r}$ is a closed set of $\R^{m \times n}$ but non-smooth in those points of rank strictly less than $r$ \cite{lee2003introduction}.
Similar to sparse recovery, the smooth part of $\mathcal{M}_{\leq r}$ is the set of matrices of fixed rank $r$:
\begin{align*}
    \mathcal{M}_{= r} = \{ \bm X \in \R^{m \times n} \mid \rank(\bm X) = r \} .
\end{align*}
At any $\bm X^* \in \mathcal{M}_{= r}$, it is shown \cite{vu2021perturbation} that derivative of $\P_{\mathcal{M}_{\leq r}}$ is a linear mapping from $\R^{m \times n}$ to $\R^{m \times n}$ satisfying
\begin{align} \label{equ:dP_rank}
    \nabla \P_{\mathcal{M}_{\leq r}} (\bm X^*) (\bm \Delta) = \bm \Delta - \bm P_{\bm U_\perp} \bm \Delta \bm P_{\bm V_\perp} ,
\end{align}
where $\bm P_{\bm U_\perp}$ and $\bm P_{\bm V_\perp}$ are the projections onto the left and right null spaces of $\bm X^*$, respectively. 
More importantly, for any $\bm \Delta \in \R^{m \times n}$, Theorem~3 in \cite{vu2021perturbation} asserts that
\begin{align*}
    \sup_{\bm Y \in \Pi_{\mathcal{M}_{\leq r}} (\bm X^* + \bm \Delta)} &\unorm{\bm Y - \bm X^* - \nabla \P_{\mathcal{M}_{\leq r}} (\bm X^*) (\bm \Delta)}_F \\
    &\qquad \qquad \qquad \leq \frac{4(1+\sqrt{2})}{\sigma_r(\bm X^*)} \unorm{\bm \Delta}_F^2 . \numberthis \label{equ:perturb}
\end{align*}

\subsubsection{Vectorized version of matrix completion}

To apply our proposed framework to matrix completion, we consider a vectorized version of (\ref{prob:mcp}) as follows.
Slightly extending the notation, we denote $\C = \{ \vect(\bm X) \mid \bm X \in \mathcal{M}_{\leq r} \}$ and $\vect(\Omega) = \{ (j-1)m+i \mid (i,j) \in \Omega \}$ with $s$ distinct elements $1 \leq i_1 < \ldots < i_s \leq mn$.
Let $\bm S_{\Omega} = [\bm e_{i_1}, \ldots, \bm e_{i_s}] \in \R^{mn \times s}$ be the selection matrix satisfying
\begin{align*}
\begin{cases}
    \bm S_{\Omega}^{\topnew} \bm S_{\Omega} = \bm I_{s} , \\
    \vect\bigl( \P_{{\Omega}} (\bm X) \bigr) = \bm S_{\Omega} \bm S_{\Omega}^{\topnew} \vect(\bm X) .
\end{cases}
\end{align*}
Then, problem (\ref{prob:mcp}) can be represented as
\begin{align*}
    \min_{\bm x \in \R^{mn}} \frac{1}{2} \unorm{\bm S_{\Omega} \bm S_{\Omega}^\topnew \bm x - \bm S_{\Omega} \bm S_{\Omega}^\topnew \vect(\bm M)}^2 \text{ s.t. } \bm x \in \C .
\end{align*}

\noindent \textbf{Step 1:} 
In this vectorized version of matrix completion, we have $\bm A = \bm S_{\Omega} \bm S_{\Omega}^\topnew$, $\bm b = \bm S_{\Omega} \bm S_{\Omega}^\topnew \vect(\bm M)$, and $\C$ is a closed non-convex set. 
For any vector $\bm x \in \R^{mn}$, let $\bm X = \vect^{-1}(\bm x)$ with $\P_{\mathcal{M}_{\leq r}} (\bm X) = \sum_{i=1}^r \sigma_i(\bm X) \bm u_i(\bm X) \bm v_i(\bm X)^\topnew$, for some $(\bm \Sigma, \bm U, \bm V) \in \text{SVD}(\bm X)$. 
The projection $\P_\C$ is given by $\P_\C (\bm x) = \vect ( \sum_{i=1}^r \sigma_i(\bm X) \bm u_i(\bm X) \bm v_i(\bm X)^\topnew )$.
Using the fact that $\vect(\bm u \bm v^\topnew) = \bm v \otimes \bm u$, for any vectors $\bm u$ and $\bm v$ of compatible dimensions, $\P_\C$ can then be represented as 
\begin{align} \label{equ:mcp_Pc_vec}
   \P_\C (\bm x) = \sum_{i=1}^r \sigma_i(\bm X) \bigl( \bm v_i(\bm X) \otimes \bm u_i(\bm X) \bigr) .
\end{align}

\noindent \textbf{Step 2:}
In the following, we show that $\P_\C$ is Lipschitz-continuously differentiable at any point in the set
\begin{align*}
    \C_{=r} = \{ \vect(\bm X) \mid \bm X \in \mathcal{M}_{= r} \} .
\end{align*}
In particular, for any $\bm x^* \in \C_{=r}$, we prove that $\P_\C$ is Lipschitz-continuously differentiable at $\bm x^*$ with
\begin{align*}
    &\nabla \P_\C (\bm x^*) = \bm P_{\bm U_\perp \bm V_\perp}^\perp , \\
    &c_1(\bm x^*)=\infty , \quad c_2(\bm x^*)=\frac{4(1+\sqrt{2})}{\sigma_r(\bm X^*)} ,
\end{align*}
where $\bm X^*=\vect^{-1}(\bm x^*)$.
Indeed, the constants $c_1(\bm x^*)$ and $c_2(\bm x^*)$ are obtained from the matrix inequality form (\ref{equ:perturb}).
Regarding $\nabla \P_\C (\bm x^*)$, let $\bm P_{\bm U_\perp}$ and $\bm P_{\bm V_\perp}$ be the projections onto the left and right null spaces of $\bm X^*$, respectively.  
Denote $\bm P_{\bm U_\perp \bm V_\perp} = \bm P_{\bm V_\perp} \otimes \bm P_{\bm U_\perp}$ and $\bm P_{\bm U_\perp \bm V_\perp}^\perp = \bm I_{mn} - \bm P_{\bm U_\perp \bm V_\perp}$.
Since $\vect (\bm A \bm B \bm C) = (\bm C^\topnew \otimes \bm A) \vect(\bm B)$, for any matrices $\bm A$, $\bm B$, and $\bm C$ of compatible dimensions, (\ref{equ:dP_rank}) can be vectorized to obtain $\nabla \P_\C (\bm x^*) (\bm \delta) = (\bm I_{mn} - \bm P_{\bm V_\perp} \otimes \bm P_{\bm U_\perp}) \bm \delta = \bm P_{\bm U_\perp \bm V_\perp}^\perp \bm \delta$ for any $\bm \delta \in \R^{mn}$.

Next, the stationarity condition (\ref{equ:stationary}) can be represented using $\nabla \P_\C (\bm x^*) (\bm \delta) = \bm P_{\bm U_\perp \bm V_\perp}^\perp \bm \delta$ as $\bm P_{\bm U_\perp \bm V_\perp}^\perp \bm S_{\Omega} \bm S_{\Omega}^{\topnew} \bigl ( \bm S_{\Omega} \bm S_{\Omega}^{\topnew} \bm x^* - \bm S_{\Omega} \bm S_{\Omega}^{\topnew} \vect (\bm M) \bigr) = \bm 0$.
Denote $\bm Q_{\perp} \in \R^{mn \times r(m+n-r)}$ the matrix satisfying $\bm Q_{\perp}^\topnew \bm Q_{\perp} = \bm I_{r(m+n-r)}$ and $\bm Q_{\perp} \bm Q_{\perp}^\topnew = \bm P_{\bm U_\perp \bm V_\perp}^\perp$. Then, we obtain the conditions for $\bm x^*$ to be a Lipschitz stationary point of (\ref{prob:mcp}) are $\bm x^* \in \C_{=r}$ and
\begin{align} \label{equ:QS}
    \bm Q_{\perp}^\topnew \bm S_{\Omega} \bm S_{\Omega}^{\topnew} \bigl( \bm x^* - \vect (\bm M) \bigr) = \bm 0 .
\end{align}

\noindent \textbf{Step 3:}
The stationarity condition (\ref{equ:QS}) leads to two cases. The first case is when $\bm Q_{\perp}^\topnew \bm S_{\Omega}$ has full (row-)rank and hence,
\begin{align} \label{equ:mcp_restrict}
    \bm S_{\Omega}^{\topnew} \bigl( \bm x^* - \vect (\bm M) \bigr) = \bm 0 .
\end{align}
In matrix form, (\ref{equ:mcp_restrict}) can be rewritten as $\P_\Omega(\bm X^*) = \P_\Omega(\bm M)$, which implies $\bm X^*$ is a global minimizer of (\ref{prob:mcp}).
Interestingly, this case enjoys the special setting considered in Corollary~\ref{cor:noiseless} as
\begin{align} \label{equ:mcp_zx}
    \bm z^*_\eta = \bm x^* - \eta \bm S_{\Omega} \bm S_{\Omega}^{\topnew} \bigl(\bm x^* - \vect(\bm M) \bigr) = \bm x^* ,
\end{align}
for any $\eta>0$. 
In the second case, if $\bm Q_{\perp}^\topnew \bm S_{\Omega}$ has rank strictly less than $r(m+n-r)$, then $\bm S_{\Omega}^{\topnew} \bigl( \bm x^* - \vect (\bm M) \bigr)$ may not be $\bm 0$ (e.g., a non-zero right singular vector of $\bm Q_{\perp}^\topnew \bm S_{\Omega}$). 
This implies $\bm z^*_\eta \neq \bm x^*$ and one needs to characterize the Lipschitz-continuous differentiability of the projection $\P_\C$ onto the set of low-rank matrices at $\bm z^*_\eta$ that may not have exact rank $r$. 
While the derivative of $\P_\C$ at a matrix with rank greater than $r$ has been studied in \cite{feppon2018geometric,vu2021perturbation}, it requires complete development of the error bound on the first-order expansion of this operator to obtain the constants $c_1(\bm z^*_\eta)$ and $c_2(\bm z^*_\eta)$.
For the purpose of demonstration, we restrict our subsequent analysis to the first case when $\bm Q_{\perp}^\topnew \bm S_{\Omega}$ has full (row-)rank.
Since $\bm z^*_\eta = \bm x^*$ in this case, $\P_\C$ is Lipschitz-continuously differentiable at $\bm z^*_\eta$ with
\begin{align*}
    \nabla \P_\C (\bm z^*_\eta) = \bm P_{\bm U_\perp \bm V_\perp}^\perp ,~c_1(\bm z^*_\eta) = \infty ,~ c_2(\bm z^*_\eta)=\frac{4(1+\sqrt{2})}{\sigma_r(\bm X^*)} .
\end{align*}

\noindent \textbf{Step 4:} 
Since $\nabla \P_{\C} (\bm z^*_\eta) = \nabla \P_\C (\bm x^*) = \bm P_{\bm U_\perp \bm V_\perp}^\perp$, using (\ref{equ:rho_noiseless}), we obtain the asymptotic linear rate as 
\begin{align} \label{equ:rho_MCP_1}
    \rho = \max \{ \uabs{1 - \eta \lambda_1} , \uabs{1 - \eta \lambda_{r(m+n-r)}} \} ,
\end{align}
where $\lambda_1$ and $\lambda_{r(m+n-r)}$ are the largest and smallest eigenvalues of $\bm Q_{\perp}^\topnew \bm S_\Omega \bm S_{\Omega}^\topnew \bm Q_{\perp}$, respectively.

\noindent \textbf{Step 5:} 
From (\ref{equ:rho_MCP_1}), we have $\rho<1$ if and only if $\bm Q_{\perp}^\topnew \bm S_\Omega \bm S_{\Omega}^\topnew \bm Q_{\perp}$ has full rank and 
\begin{align*}
    0 < \eta < \frac{2}{\unorm{\bm Q_{\perp}^\topnew \bm S_\Omega \bm S_{\Omega}^\topnew \bm Q_{\perp}}_2} .
\end{align*}
Here, we would like to point out the condition $\bm Q_{\perp}^\topnew \bm S_\Omega \bm S_{\Omega}^\topnew \bm Q_{\perp}$ has full rank implies $s \geq r(m+n-r)$, which can be interpreted as a requirement for the number of observations being no less than the degree of freedom in matrix completion. The invertibility of $\bm Q_{\perp}^\topnew \bm S_\Omega \bm S_{\Omega}^\topnew \bm Q_{\perp}$ is also equivalent to the injectivity of the sampling operator restricted to the tangent space $T$ to $\mathcal{M}_{\leq r}$ at $\bm X^*$, denoted by $\mathcal{A}_{\Omega T}$ in \cite{candes2009exact}-Section~4.2. It is interesting to note that under the standard assumptions on uniform sampling and incoherence property, Cand{\`e}s and Recht \cite{candes2009exact} showed that $\mathcal{A}_{\Omega T}$ is injective with high probability.

\noindent \textbf{Step 6:} 
Recall that $c_1(\bm x^*) = c_1(\bm z^*_\eta) = \infty$. Since $\nabla \P_{\C} (\bm z^*_\eta) = \nabla \P_\C (\bm x^*) = \bm P_{\bm U_\perp \bm V_\perp}^\perp$, using (\ref{cond:roc_sym}), the region of linear convergence is given by
\begin{align} \label{equ:roc_MCP}
    \unorm{\bm x - \bm x^*} \leq \frac{(1 - \rho)\sigma_r(\bm X^*)}{8(1+\sqrt{2})} .
\end{align}

\begin{figure}[t]
    \centering
    \includegraphics[scale=.61]{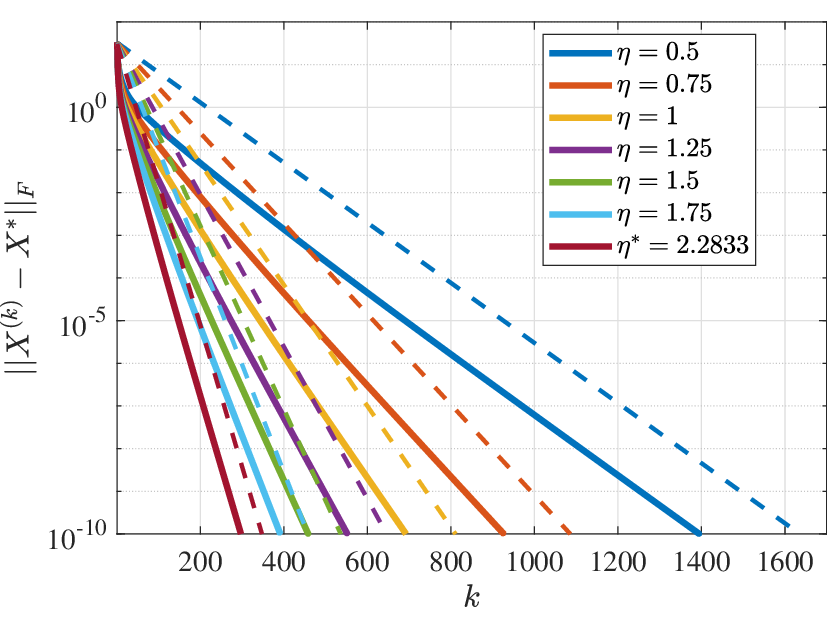}
    \caption{(Log-scale) plot of the distance between the current iterate and the local minimizer of the matrix completion problem, as a function of the number of iterations. Each solid line corresponds to PGD with a different fixed step size. Each dashed line represents the respective exponential bound $\rho^k$ up to a constant, where the theoretical rate $\rho$ is given by (\ref{equ:rho_MCP_1}). In the experiment, we select $m=50, n=40, r=3$, and $s=800$. The optimal step size $\eta_{opt}=2.2833$ is given by (\ref{equ:mcp_opt}), with the corresponding optimal rate $\rho_{opt}=0.9265$.}
    \label{fig:mcp}
\end{figure}

\begin{remark}
Similar to (\ref{equ:CLS_eta_opt}) and (\ref{equ:CLS_rho_opt}), the optimal step size and the optimal convergence rate are given by
\begin{align*}
    \eta_{opt} &= \frac{2}{\lambda_{1}(\bm Q_{\perp}^\topnew \bm S_\Omega \bm S_{\Omega}^\topnew \bm Q_{\perp})+\lambda_{r(m+n-r)}(\bm Q_{\perp}^\topnew \bm S_\Omega \bm S_{\Omega}^\topnew \bm Q_{\perp})} , \\
    \rho_{opt} &= 1 - \frac{2}{\kappa(\bm Q_{\perp}^\topnew \bm S_\Omega \bm S_{\Omega}^\topnew \bm Q_{\perp}) + 1} . \numberthis \label{equ:mcp_opt}
\end{align*}
We consider the following numerical experiment to verify the analytical rate in (\ref{equ:rho_MCP_1}). The data is generated randomly as follows. First, we sample two matrices $\bm A$ and $\bm B$ with $i.i.d$ normally distributed entries, of dimensions $50 \times 3$ and $40 \times 3$, respectively. Next, we obtain the rank-$3$ matrix of dimension $50 \times 40$ as the product $\bm X^* = \bm A \bm B^\topnew$. Third, we select $800$ observations uniformly at random among the $2000$ positions in $\bm X^*$. 
We apply PGD with different step sizes (listed in Fig.~\ref{fig:mcp}) including $\eta_{opt}$ in (\ref{equ:mcp_opt}) and record the value of $\norm{\bm X^{(k)} - \bm X^*}_F$ as a function of $k$.
It can be seen from Fig.~\ref{fig:mcp} that the theoretical rate matches well the empirical rate, reassuring the correctness of our analysis in the previous section. 
\end{remark}

\begin{remark}
The rate in (\ref{equ:rho_MCP_1}) has not been proposed in the literature. However, in the special case of using unit step size, it matches the rate established for the IHTSVD algorithm in \cite{chunikhina2014performance}. In their work, the authors provide the result relative to the matrix $(\bm S_{\bar{\Omega}})^\topnew \bm P_{\bm U_\perp \bm V_\perp} \bm S_{\bar{\Omega}}$ instead of $\bm Q_{\perp}^\topnew \bm S_\Omega \bm S_{\Omega}^\topnew \bm Q_{\perp}$, where $\bm S_{\bar{\Omega}} \in \R^{mn \times (mn-s)}$ is the selection matrix that is complement to $\bm S_\Omega$. 
It can be shown that the two matrices share the same set of eigenvalues while may only differ by the eigenvalues at $1$. 
Since IHTSVD uses $\eta=1$, these unit eigenvalues do not affect the maximization in (\ref{equ:rho_MCP_1}). Compared to the local convergence result in \cite{chunikhina2014performance}, our application in this subsection not only considers PGD with different step sizes but also includes the region of linear convergence in (\ref{equ:roc_MCP}).
\end{remark}

\section{Conclusion and Future Work}
\label{sec:conclusion}

We presented a unified framework to analyze the local convergence of projected gradient descent for constrained least squares. Our analysis provides the asymptotic rate of convergence in a closed-form expression, the number of iterations required to reach certain accuracy, and the local region of convergence.
Notably, our technique relies on the Lipschitz-continuous differentiability of the projection operator at two key points: $\bm x^*$ and $\bm z^*_\eta$.
Finally, we demonstrated the application of our proposed framework to local convergence analysis of PGD in four well-known problems: linear equality-constrained least squares, sparse recovery, least squares with a unit norm constraint, and matrix completion.

While the work here focuses on the specific setting of linear converges of the PGD algorithm, we believe it can be expanded in several directions. First, our framework can be utilized to analyze the following cases: {\em(i)} adaptive step size schemes (e.g., the backtracking line search rule), {\em(ii)} accelerated methods (e.g., the Nesterov's accelerated gradient and the Heavy Ball method), {\em(iii)} general objective functions other than least squares, and {\em(iv)} other algorithms for manifold optimization such as Riemannian gradient descent. 
Another interesting research direction is to sharpen the theoretical bound on the ROC in order to better explain the actual region in which the algorithm converges to the desired solution.
Finally, the proposed framework can be used to further study the performance of PGD for a variety of constrained least squares problems arising in the area of phase-only beamforming \cite{tranter2017fast}, online power system optimization \cite{hauswirth2016projected}, spectral compressed sensing \cite{cai2018spectral}, and linear dimensionality reduction \cite{cunningham2015linear}.

\appendices

\section{Proof of Lemma~\ref{lem:fixed}}
\label{appdx:fixed}

Our goal in the proof of Lemma~\ref{lem:fixed} is to show that if the fixed point condition $\bm x^* = \P_{\C} (\bm x^* - \eta \bm A^\topnew ( \bm A \bm x^* - \bm b ))$ holds, then the stationarity condition $\nabla \P_\C (\bm x^*) \bm A^\topnew\bigl( \bm A \bm x^* - \bm b \bigr) = \bm 0$ holds.
Note that if $\bm A^\topnew ( \bm A \bm x^* - \bm b ) = \bm 0$, then the stationarity condition holds trivially. Hence, we focus on the proof for $\bm A^\topnew ( \bm A \bm x^* - \bm b ) \neq \bm 0$.
We first show that for any $0 \leq \alpha < 1$, $\bm x^* = \P_{\C} (\bm x^* - \eta \bm A^\topnew ( \bm A \bm x^* - \bm b ))$ is a sufficient condition for
\begin{align} \label{equ:alpha}
    \Pi_\C \bigl(\bm x^* - \alpha \eta \bm A^\topnew ( \bm A \bm x^* - \bm b ) \bigr) = \{\bm x^*\} .
\end{align}
Then, using (\ref{equ:alpha}) and the differentiability of $\P_\C$ at $\bm x^*$, we prove that $\nabla \P_\C (\bm x^*) \bm A^\topnew\bigl( \bm A \bm x^* - \bm b \bigr) = \bm 0$. We proceed with the detailed proof.

First, let $\bm v^*=\bm A^\topnew ( \bm A \bm x^* - \bm b )$ and $\bm z^*_{\alpha \eta} = \bm x^* - \alpha \eta \bm v^*$.
On the one hand, for any $0 \leq \alpha < 1$ and $\bm y \in \Pi_\C(\bm z^*_{\alpha \eta})$, we have
\begin{align*}
    \unorm{\bm y - \bm z^*_\eta} &= \unorm{(\bm y - \bm z^*_{\alpha \eta}) + (\bm z^*_{\alpha \eta} - \bm z^*_\eta)} \\
    &\leq \unorm{\bm y - \bm z^*_{\alpha \eta}} + \unorm{\bm z^*_{\alpha \eta} - \bm z^*_\eta} \\
    &= d(\bm z^*_{\alpha \eta}, \C)  + \unorm{\bm z^*_{\alpha \eta} - \bm z^*_\eta} \\
    &\leq \unorm{\bm x^* - \bm z^*_{\alpha \eta}} + \unorm{\bm z^*_{\alpha \eta} - \bm z^*_\eta} , \numberthis \label{equ:dyz00}
\end{align*}
where the first inequality uses the triangle inequality that holds when $\bm y - \bm z^*_{\alpha \eta} = \beta (\bm z^*_{\alpha \eta} - \bm z^*_\eta)$, for some $\beta \geq 0$.
Using the fact that $\bm z^*_{\eta} = \bm x^* - \eta \bm v^*$ and $\bm z^*_{\alpha \eta} = \bm x^* - \alpha \eta \bm v^*$, we obtain
\begin{align*}
    \unorm{\bm x^* - \bm z^*_{\alpha \eta}} + \unorm{\bm z^*_{\alpha \eta} - \bm z^*_\eta} &= \unorm{\alpha \eta \bm v^*} + \unorm{(1-\alpha) \eta \bm v^*} \\
    &= \unorm{\alpha \eta \bm v^* + (1-\alpha) \eta \bm v^*} \\
    &= \unorm{ \eta \bm v^*} \\
    &= \unorm{\bm x^* - \bm z^*_\eta} . \numberthis \label{equ:dyz01}
\end{align*}
From (\ref{equ:dyz00}) and (\ref{equ:dyz01}), we have
\begin{align} \label{equ:dyz}
    \unorm{\bm y - \bm z^*_\eta} \leq \unorm{\bm x^* - \bm z^*_\eta} ,
\end{align}
with the equality holding if and only if
\begin{align} \label{equ:dyz_equal}
    \begin{cases}
        \bm y - \bm z^*_{\alpha \eta} = \beta (\bm z^*_{\alpha \eta} - \bm z^*_\eta) ,  \\
        \unorm{\bm y - \bm z^*_{\alpha \eta}} = \unorm{\bm x^* - \bm z^*_{\alpha \eta}} .
    \end{cases} 
\end{align}
Using the fact that $\bm z^*_{\eta} = \bm x^* - \eta \bm v^*$ and $\bm z^*_{\alpha \eta} = \bm x^* - \alpha \eta \bm v^*$, (\ref{equ:dyz_equal}) holds if and only if
\begin{align} \label{equ:dyz_equal2}
    \beta = \frac{\alpha}{1-\alpha} \text{ and } \bm y = \bm x^* .
\end{align}
On the other hand, since $\bm x^* = \P_\C(\bm z^*_\eta)$ and $\bm y \in \C$, we have 
\begin{align} \label{equ:dyz2}
    \unorm{\bm y - \bm z^*_\eta} \geq \unorm{\bm x^* - \bm z^*_\eta} .
\end{align}
From (\ref{equ:dyz}) and (\ref{equ:dyz2}), we conclude that $\unorm{\bm y - \bm z^*_\eta} = \unorm{\bm x^* - \bm z^*_\eta}$.
Moreover, from (\ref{equ:dyz_equal2}), the equality holds if and only if $\bm y = \bm x^*$.
Since this holds for any $\bm y \in \Pi_\C(\bm z^*_{\alpha \eta})$, we conclude that $\Pi_\C(\bm z^*_{\alpha \eta}) = \{ \bm x^* \}$ for all $0 \leq \alpha < 1$. 

Next, using the differentiability of the projection $\P_\C$ at $\bm x^*$ from Definition~\ref{def:Pc_diff}, we have
\begin{align*}
    \lim_{\alpha \to 0} \sup_{\bm y \in \Pi_\C(\bm x^* - \alpha \eta \bm v^*)} \frac{\norm{\bm y - \P_\C(\bm x^*) - \nabla \P_\C(\bm x^*) (\alpha \eta \bm v^*)}}{\norm{\alpha \eta \bm v^*}} = 0 .
\end{align*}
Substituting $\Pi_\C(\bm x^* - \alpha \eta \bm v^*) = \Pi_\C(\bm z^*_{\alpha \eta}) = \{ \bm x^* \}$ and $\P_\C(\bm x^*) = \bm x^*$ into the last equation, we obtain
\begin{align*}
    0 &= \lim_{\alpha \to 0} \frac{\norm{\bm x^* - \bm x^* - \alpha \eta \nabla \P_\C(\bm x^*) \bm v^*}}{\norm{\alpha \eta \bm v^*}} \\
    &= \lim_{\alpha \to 0} \frac{\alpha \eta \norm{\nabla \P_\C(\bm x^*) \bm v^*}}{\alpha \eta \norm{\bm v^*}} \\
    &= \frac{\norm{\nabla \P_\C(\bm x^*) \bm v^*}}{\norm{\bm v^*}} ,
\end{align*}
which only holds if $\nabla \P_\C(\bm x^*) \bm v^* = \bm 0$. 
This completes our proof of the lemma.

\section{Proof of Corollary~\ref{cor:noiseless}}
\label{appdx:noiseless}

In the following, under the assumption $\nabla \P_{\C} (\bm z^*_\eta) = \nabla \P_\C (\bm x^*) = \bm U_{\bm x^*} \bm U_{\bm x^*}^\topnew$, we show that {\em(i)} the asymptotic convergence rate $\rho(\bm H)$ is given by (\ref{equ:rho_noiseless}), {\em(ii)} the sufficient conditions for $\rho(\bm H)<1$ are $(\bm A \bm U_{\bm x^*})^\topnew \bm A \bm U_{\bm x^*}$ is full rank and (\ref{equ:eta2}) holds, and {\em(iii)} the region of linear convergence can be simplified from (\ref{cond:roc}) to (\ref{cond:roc_sym}).

First, we prove (\ref{equ:rho_noiseless}) by simplifying the expression of $\bm H$ in (\ref{equ:PZT}) and the fact that $\bm U_{\bm x^*}^\topnew \bm U_{\bm x^*} = \bm I_d$ as follows. Substituting $\nabla \P_\C (\bm x^*)$ and $\nabla \P_\C (\bm z^*_\eta)$ by $\bm U_{\bm x^*} \bm U_{\bm x^*}^\topnew$ into (\ref{equ:H}) yields
\begin{align*}
    \bm H &= \bm U_{\bm x^*} \bm U_{\bm x^*}^\topnew (\bm I_n - \eta \bm A^\topnew \bm A) \bm U_{\bm x^*} \bm U_{\bm x^*}^\topnew \\
    &= \bm U_{\bm x^*} ( \bm I_d - \eta \bm U_{\bm x^*}^\topnew \bm A^\topnew \bm A \bm U_{\bm x^*} ) \bm U_{\bm x^*}^\topnew ,
\end{align*}
where the second equality stems from $\bm U_{\bm x^*}^\topnew \bm U_{\bm x^*} = \bm I_d$. Since $\bm H$ is symmetric, its spectral radius equals to its spectral norm:
\begin{align*}
    \rho(\bm H) &= \unorm{\bm U_{\bm x^*} ( \bm I_d - \eta \bm U_{\bm x^*}^\topnew \bm A^\topnew \bm A \bm U_{\bm x^*} ) \bm U_{\bm x^*}^\topnew}_2 .
\end{align*}
Using the fact that the spectral norm is invariant under left-multiplication by matrices with orthonormal columns and right-multiplication by matrices with orthonormal rows (see \cite{meyer2000matrix} - Exercise~5.6.9), we further have
\begin{align} \label{equ:H_eigen}
    \rho(\bm H) &= \unorm{\bm I_d - \eta \bm U_{\bm x^*}^\topnew \bm A^\topnew \bm A \bm U_{\bm x^*}}_2 .
\end{align}
Let $\bm U_{\bm x^*}^\topnew \bm A^\topnew \bm A \bm U_{\bm x^*} = \hat{\bm U} \hat{\bm \Lambda} \hat{\bm U}^\topnew$ be an eigendecomposition, where $\hat{\bm U} \in \R^{d \times d}$ is an orthogonal matrix and $\hat{\bm \Lambda} = \diag(\lambda_1(\bm U_{\bm x^*}^\topnew \bm A^\topnew \bm A \bm U_{\bm x^*}), \ldots, \lambda_d(\bm U_{\bm x^*}^\topnew \bm A^\topnew \bm A \bm U_{\bm x^*}))$.
Since $\hat{\bm U}^\topnew \hat{\bm U} = \bm I_d$, (\ref{equ:H_eigen}) can be represented as
\begin{align*}
    \rho(\bm H) &= \norm{\hat{\bm U} ( \bm I_d - \eta \hat{\bm \Lambda} ) \hat{\bm U}^\topnew}_2 = \norm{\bm I_d - \eta \hat{\bm \Lambda}}_2 .
\end{align*}
Now using the fact that the spectral norm of a diagonal matrix is the maximum of the absolute values of its diagonal entries, we obtain
\begin{align*}
    \rho(\bm H) &= \max_{1 \leq i\leq d} \uabs{1 - \eta \lambda_i(\bm U_{\bm x^*}^\topnew \bm A^\topnew \bm A \bm U_{\bm x^*})} \\
    &= \max \{ \uabs{1-\eta \lambda_1} , \uabs{1-\eta \lambda_d} \}.
\end{align*}

Second, we establish the sufficient conditions for $\rho(\bm H)<1$ by bounding each term inside the maximum in (\ref{equ:rho_noiseless}) as follows. Since $\lambda_d \geq 0$, we have
\begin{align*}
    -1 < 1 - \eta \lambda_1 \leq 1 - \eta \lambda_d \leq 1 , \text{ for } i=1,\ldots,d,
\end{align*}
if $0 < \eta < 2/\lambda_1$.
It is also noted from the definition of the spectral norm that $\norm{\bm A \bm U_{\bm x^*}}_2^2 = \lambda_1$.
Therefore, $\rho(\bm H) \leq 1$ provided that (\ref{equ:eta2}) holds.
The equality $\rho(\bm H)=1$ holds if and only if $\lambda_d=0$, i.e., $(\bm A \bm U_{\bm x^*})^\topnew \bm A \bm U_{\bm x^*}$ is singular.
In other words, when $(\bm A \bm U_{\bm x^*})^\topnew \bm A \bm U_{\bm x^*}$ is full rank and (\ref{equ:eta2}) holds, the linear convergence is guaranteed as $\rho(\bm H)<1$. 

Finally, the region of linear convergence in (\ref{cond:roc_sym}) is determined based on simplifying (\ref{cond:roc}) as follows. First, using Remark~\ref{rmk:kappa}, we obtain $\kappa(\bm Q)=1$. Second, from (\ref{equ:PZT}), we have $\unorm{\nabla \P_\C (\bm z^*_\eta)}_2 = \unorm{\P_{T_{\bm x^*}(\C)}}_2 = 1$.
Third, substituting $\kappa(\bm Q)=1$ and $\unorm{\nabla \P_\C (\bm z^*_\eta)}_2=1$ into (\ref{cond:roc}) yields (\ref{cond:roc_sym}).
This completes our proof of the corollary.

\section{Proof of Lemma~\ref{lem:delta_H}}
\label{appdx:delta_H}

Our goal is to show the error vector $\bm \delta^{(k)}$ satisfies the asymptotically-linear quadratic system dynamic in (\ref{equ:delta_H}) and to bound the norm of the residual $\bm q_2$ by (\ref{equ:q_hat}). 

First, our key idea in proving (\ref{equ:delta_H}) is the Lipschitz-continuous differentiability of $\P_\C$ at $\bm x^*$ and at $\bm z^*_\eta$.
Specifically, for any $k$ such that $\bm \delta^{(k)}$ admits a  perturbation $(\bm I - \eta \bm A^\topnew \bm A) \bm \delta^{(k)}$ that satisfies
\begin{align} \label{cond:u_eta}
    \unorm{(\bm I - \eta \bm A^\topnew \bm A) \bm \delta^{(k)}} < c_1(\bm z^*_\eta) ,
\end{align}
applying the Lipschitz-continuous differentiability of $\P_\C$ at $\bm z^*_\eta$ to (\ref{equ:delta_2}) yields
\begin{align*} 
    \bm \delta^{(k+1)} = \nabla \P_{\C} (\bm z^*_\eta) &(\bm I - \eta \bm A^\topnew \bm A) \bm \delta^{(k)} + \bm q_1(\bm \delta^{(k)}) , \numberthis \label{equ:delta}
\end{align*}
where the residual $\bm q_1: \R^n \to \R^n$ satisfies
\begin{align*}
    \unorm{\bm q_1(\bm \delta^{(k)})} &\leq c_2(\bm z^*_\eta) \unorm{(\bm I - \eta \bm A^\topnew \bm A) \bm \delta^{(k)}}^2 \\
    &\leq c_2(\bm z^*_\eta) u_\eta^2 \unorm{\bm \delta^{(k)}}^2 . \numberthis \label{equ:q1} 
\end{align*}
On the other hand, using the fact that $\bm x^* = \P_\C(\bm x^*)$, $\bm x^{(k)} = \P_\C(\bm x^{(k)})$, and the Lipschitz-continuous differentiability of $\P_\C$ at $\bm x^*$ with the perturbation $\bm \delta^{(k)} \in \B(\bm 0, c_1(\bm x^*))$, we obtain
\begin{align*}
    \bm \delta^{(k)} &= \bm x^{(k)} - \bm x^* \\
    &= \P_\C (\bm x^* + \bm \delta^{(k)}) - \P_\C (\bm x^*) \\
    &= \nabla \P_\C (\bm x^*) \bm \delta^{(k)} + \bm q_{\bm x^*}(\bm \delta^{(k)}) , \numberthis \label{equ:delta_T}
\end{align*}
where the residual $\bm q_{\bm x^*}: \R^n \to \R^n$ satisfies $\unorm{\bm q_{\bm x^*}(\bm \delta^{(k)})} \leq c_2(\bm x^*) \unorm{\bm \delta^{(k)}}^2$.
We proceed with the proof of (\ref{equ:delta_H}) by combining the results from (\ref{equ:delta}) and (\ref{equ:delta_T}) as follows.
Since $\unorm{(\bm I - \eta \bm A^\topnew \bm A) \bm \delta^{(k)}} \leq \unorm{\bm I - \eta \bm A^\topnew \bm A}_2 \unorm{\bm \delta^{(k)}} = u_\eta \unorm{\bm \delta^{(k)}}$, the sufficient condition for (\ref{cond:u_eta}) is $\unorm{\bm \delta^{(k)}} < c_1(\bm z^*_\eta)/u_\eta$. Thus, $\unorm{\bm \delta^{(k)}} < \min\{c_1(\bm x^*), c_1(\bm z^*_\eta)/u_\eta \}$ is sufficient for both (\ref{equ:delta}) and (\ref{equ:delta_T}).
Substituting (\ref{equ:delta_T}) into the RHS of (\ref{equ:delta}), we obtain (\ref{equ:delta_H}) with $\bm q_2(\bm \delta^{(k)}) = \nabla \P_{\C} (\bm z^*_\eta) (\bm I - \eta \bm A^\topnew \bm A) \bm q_{\bm x^*}(\bm \delta^{(k)}) + \bm q_1(\bm \delta^{(k)})$.
Next, to bound the norm of the residual $\bm q_2$, we apply the triangle inequality as follows
\begin{align*} 
    \unorm{\bm q_2(\bm \delta^{(k)})} \leq &\unorm{\nabla \P_{\C} (\bm z^*_\eta) (\bm I - \eta \bm A^\topnew \bm A) \bm q_{\bm x^*}(\bm \delta^{(k)})} \\
    &\qquad + \unorm{\bm q_1(\bm \delta^{(k)})} . \numberthis \label{equ:bound1}
\end{align*}
On the one hand, the first term on the RHS of (\ref{equ:bound1}) can be bounded by
\begin{align*}
    &\unorm{\nabla \P_{\C} (\bm z^*_\eta) (\bm I - \eta \bm A^\topnew \bm A) \bm q_{\bm x^*}(\bm \delta^{(k)})} \\
    &\qquad \leq \unorm{\nabla \P_{\C} (\bm z^*_\eta)}_2 \unorm{\bm I - \eta \bm A^\topnew \bm A}_2 \unorm{\bm q_{\bm x^*}(\bm \delta^{(k)})} \\
    &\qquad \leq \unorm{\nabla \P_{\C} (\bm z^*_\eta)}_2 u_\eta c_2 (\bm x^*) \unorm{\bm \delta^{(k)}}^2 . \numberthis \label{equ:bound2}
\end{align*}
On the other hand, the second term on the RHS of (\ref{equ:bound1}) can be bounded by (\ref{equ:q1}).
Combining the two bounds, we obtain (\ref{equ:q_hat}).


\section{Proof of Lemma~\ref{lem:delta_tilde}}
\label{appdx:delta_tilde}

In this section, our goal is to show the recursion on the transformed error vector (\ref{equ:delta_tilde}) holds at any $k \in \mathbb{N}$ provided that the initial error vector lies within the region of linear convergence described by (\ref{cond:roc}).
In the first step, we prove that if the current transformed error vector lies within the region of linear convergence
\begin{align} \label{cond:delta_k}
    \unorm{\tilde{\bm \delta}^{(k)}} < \min \Bigl\{ \frac{c_1(\bm x^*)}{\unorm{\bm Q}_2}, \frac{c_1(\bm z^*_\eta)}{\unorm{\bm Q}_2 u_\eta}, \frac{1-\rho(\bm H)}{q/\unorm{\bm Q^{-1}}_2} \Bigr\} .
\end{align}
then $\tilde{\bm \delta}^{(k+1)} = \bm \Lambda \tilde{\bm \delta}^{(k)} + \bm q_3 (\tilde{\bm \delta}^{(k)})$ and moreover, the next transformed error vector also lies within the region of linear convergence
\begin{align} \label{cond:delta_k1}
    \unorm{\tilde{\bm \delta}^{(k+1)}} < \min \Bigl\{ \frac{c_1(\bm x^*)}{\unorm{\bm Q}_2}, \frac{c_1(\bm z^*_\eta)}{\unorm{\bm Q}_2 u_\eta}, \frac{1-\rho(\bm H)}{q/\unorm{\bm Q^{-1}}_2} \Bigr\} .
\end{align}
Therefore, by the principle of induction, the initial condition on the transformed error vector, i.e., (\ref{cond:delta_k}) holds at $k=0$, is the sufficient condition for (\ref{equ:delta_tilde}) to hold at any $k \in \mathbb{N}$. 
In the second step, we show that (\ref{cond:delta_k}) holds at $k=0$ if the initial condition on the error vector (\ref{cond:roc}) holds and hence, completes the proof of lemma.
We proceed with our detailed proof below.

First, let us assume that (\ref{cond:delta_k}) holds.
We have
\begin{align*}
    \unorm{\bm \delta^{(k)}} &= \unorm{\bm Q \tilde{\bm \delta}^{(k)}} \leq \unorm{\bm Q}_2 \unorm{\tilde{\bm \delta}^{(k)}} \\
    &< \unorm{\bm Q}_2 \min \Bigl\{ \frac{c_1(\bm x^*)}{\unorm{\bm Q}_2}, \frac{c_1(\bm z^*_\eta)}{\unorm{\bm Q}_2 u_\eta}, \frac{1-\rho(\bm H)}{q/\unorm{\bm Q^{-1}}_2} \Bigr\} \\
    &\leq \min \Bigl\{ {c_1(\bm x^*)}, \frac{c_1(\bm z^*_\eta)}{u_\eta} \Bigr\} , \numberthis \label{equ:delta0_2}
\end{align*}
Thus, by Lemma~\ref{lem:delta_H}, we have $\bm \delta^{(k+1)} = \bm H \bm \delta^{(k)} + \bm q_2(\bm \delta^{(k})$.
Substituting $\bm H = \bm Q \bm \Lambda \bm Q^{-1}$ and multiplying both sides with $\bm Q^{-1}$ yields $\bm Q^{-1} \bm \delta^{(k+1)} = \bm \Lambda \bm Q^{-1} \bm \delta^{(k)} + \bm Q^{-1} \bm q_2 (\bm \delta^{(k)})$.
Replacing $\bm Q^{-1} \bm \delta^{(k)}$ by $\tilde{\bm \delta}^{(k)}$ and $\bm \delta^{(k)}$ by $\bm Q \tilde{\bm \delta}^{(k)}$ in the last equation, we obtain (\ref{equ:delta_tilde}), i.e., $\tilde{\bm \delta}^{(k+1)} = \bm \Lambda \tilde{\bm \delta}^{(k)} + \bm q_3 (\tilde{\bm \delta}^{(k)})$. Here, the second term $\bm q_3$ can be bounded as follows
\begin{align*}
    &\unorm{\bm q_3 \bigl(\tilde{\bm \delta}^{(k)} \bigr)} = \unorm{\bm Q^{-1} \bm q_2 (\bm Q \tilde{\bm \delta}^{(k)})} \leq \unorm{\bm Q^{-1}}_2 \unorm{\bm q_2 (\bm Q \tilde{\bm \delta}^{(k)})} \\
    &\qquad \leq \unorm{\bm Q^{-1}}_2 \bigl( c_2(\bm x^*) + \unorm{\nabla \P_{\C} (\bm z^*_\eta)}_2 c_2(\bm z^*_\eta) \bigr) \unorm{\bm Q \tilde{\bm \delta}^{(k)}}^2 \\
    &\qquad \leq \unorm{\bm Q^{-1}}_2 \bigl( c_2(\bm x^*) + \unorm{\nabla \P_{\C} (\bm z^*_\eta)}_2 c_2(\bm z^*_\eta) \bigr) \unorm{\bm Q}_2^2 \unorm{\tilde{\bm \delta}^{(k)}}^2 \\
    &\qquad = \frac{q}{\unorm{\bm Q^{-1}}_2} \unorm{\tilde{\bm \delta}^{(k)}}^2 . \numberthis \label{equ:b2}
\end{align*}
Now, taking the norms of both sides of (\ref{equ:delta_tilde}) and applying the triangle inequality yield
\begin{align*}
    \unorm{\tilde{\bm \delta}^{(k+1)}} &\leq \unorm{\bm \Lambda \tilde{\bm \delta}^{(k)}} + \unorm{\bm q_3 \bigl(\tilde{\bm \delta}^{(k)} \bigr)} \\
    &\leq \rho(\bm H) \unorm{\tilde{\bm \delta}^{(k)}} + \frac{q}{\unorm{\bm Q^{-1}}_2} \unorm{\tilde{\bm \delta}^{(k)}}^2 \\
    &< \rho(\bm H) \unorm{\tilde{\bm \delta}^{(k)}} + \bigl(1-\rho(\bm H) \bigr) \unorm{\tilde{\bm \delta}^{(k)}} \\
    &= \unorm{\tilde{\bm \delta}^{(k)}} , \numberthis \label{equ:delta_tilde_k}
\end{align*}
where the second inequality stems from $\unorm{\tilde{\bm \delta}^{(k)}} < (1-\rho(\bm H))/(q/\unorm{\bm Q^{-1}}_2)$. 
From (\ref{cond:delta_k}) and (\ref{equ:delta_tilde_k}), we conclude that (\ref{cond:delta_k1}) holds. By the principle of induction, we have (\ref{cond:delta_k}) holds for all $k \in \mathbb{N}$ provided that it holds at $k=0$, i.e.,
\begin{align} \label{cond:delta_0}
    \unorm{\tilde{\bm \delta}^{(0)}} < \min \Bigl\{ \frac{c_1(\bm x^*)}{\unorm{\bm Q}_2}, \frac{c_1(\bm z^*_\eta)}{\unorm{\bm Q}_2 u_\eta}, \frac{1-\rho(\bm H)}{q/\unorm{\bm Q^{-1}}_2} \Bigr\} .
\end{align}

Second, we prove that (\ref{cond:roc}) is sufficient for (\ref{cond:delta_0}).
Using the definition $\tilde{\bm \delta}^{(k)} = \bm Q^{-1} \bm \delta^{(k)}$, we have
\begin{align} \label{equ:delta_Q1}
    \unorm{\tilde{\bm \delta}^{(k)}} &= \unorm{\bm Q^{-1} \bm \delta^{(k)}} \leq \unorm{\bm Q^{-1}}_2 \unorm{\bm \delta^{(k)}} .
\end{align}
Upper-bounding $\unorm{\bm \delta^{(k)}}$ by the LHS of (\ref{cond:roc}) and substituting back into (\ref{equ:delta_Q1}) yield
\begin{align*}
    \unorm{\tilde{\bm \delta}^{(k)}} < \unorm{\bm Q^{-1}}_2 \min \Bigl\{ \frac{c_1(\bm x^*)}{\kappa(\bm Q)}, \frac{c_1(\bm z^*_\eta)}{\kappa(\bm Q) u_\eta}, \frac{1-\rho(\bm H)}{q} \Bigr\} .
\end{align*}
Finally, replacing $\kappa(\bm Q)$ by the product $\unorm{\bm Q}_2 \unorm{\bm Q^{-1}}_2$ and simplifying yield (\ref{cond:delta_0}).
This completes our proof of the lemma.

\section{Proof of Lemma~\ref{lem:scalar}}
\label{appdx:scalar}

In this section, we show the convergence of $\{\unorm{\tilde{\bm \delta}^{(k)}}\}_{k=0}^\infty$ using Theorem~1 in \cite{vu2021closed}. Our idea is to consider a surrogate sequence $\{a_k\}_{k=0}^\infty$ that upper-bounds $\{\unorm{\tilde{\bm \delta}^{(k)}}\}_{k=0}^\infty$:
\begin{align*}
    \begin{cases}
        a_0=\unorm{\tilde{\bm \delta}^{(0)}} , \\
        a_{k+1} = \rho(\bm H) a_k + \frac{q}{\unorm{\bm Q^{-1}}_2} a_k^2 .
    \end{cases}
\end{align*}
First, we prove by induction that 
\begin{align} \label{equ:ak}
    \unorm{\tilde{\bm \delta}^{(k)}} \leq a_k \quad \forall k \in \mathbb{N}  .
\end{align}
The base case when $k=0$ holds trivially as $a_0=\unorm{\tilde{\bm \delta}^{(0)}}$.
In the induction step, given $\unorm{\tilde{\bm \delta}^{(k)}} \leq a_k$ for some integer $k \geq 0$, we have
\begin{align*}
    \unorm{\tilde{\bm \delta}^{(k+1)}} &\leq \rho(\bm H) \unorm{\tilde{\bm \delta}^{(k)}} + \frac{q}{\unorm{\bm Q^{-1}}_2} \unorm{\tilde{\bm \delta}^{(k)}}^2 \\
    &\leq \rho a_k + \frac{q}{\unorm{\bm Q^{-1}}_2} a_k^2 \\
    &= a_{k+1} .
\end{align*}
By the principle of induction, (\ref{equ:ak}) holds for all $k \in \mathbb{N}$.
Next, applying Theorem~1 in \cite{vu2021closed}, under the condition $a_0 = \unorm{\tilde{\bm \delta}^{(0)}} < (1-\rho(\bm H))/(q/\unorm{\bm Q^{-1}}_2)$, yields $a_k \leq \tilde{\epsilon} a_0$ for any integer $k$ satisfies (\ref{equ:k_tilde}).
From (\ref{equ:ak}), we further have $\unorm{\tilde{\bm \delta}^{(k)}} \leq a_k \leq \tilde{\epsilon} a_0 = \tilde{\epsilon} \unorm{\tilde{\bm \delta}^{(0)}}$.
This completes our proof of the lemma.

\vspace{-5pt}

\section*{Acknowledgment}
The authors would like to thank Prof. Amir Beck of Tel-Aviv University for discussion on projected gradient descent.

\ifCLASSOPTIONcaptionsoff
  \newpage
\fi



%
\bibliographystyle{IEEEtran}
\bibliography{IEEEabrv,refs}

\vfill



\clearpage
\pagenumbering{arabic}
\twocolumn[{\Large \bf Supplementary Material - ``On Local Linear Convergence of Projected Gradient Descent for Constrained Least Squares'', Trung Vu and Raviv Raich}\\ \\]

\renewcommand{\appendixname}{Section}
\appendices

\section{Related Works}

In this section, we review existing approaches to convergence analysis of iterative first-order methods in optimization including projected gradient descent. We present several aspects of convergence, namely, convergence to a global versus a local optimum and speed of convergence. Finally, we clarify our contribution in this work with regard to previous works in the literature.

\subsection{Convergence of Iterative First-Order Methods} 
Convergence properties of iterative algorithms such as PGD often involve two key aspects: the quality of convergent points and the speed of convergence.
On the one hand, the quality of convergent points provides useful insights into when the algorithm converges, whether it converges to a stationary point or a set of stationary points of the problem, and how big is the gap between the objective function at the convergent point and the optimal objective value. 
On the other hand, the speed of convergence concerns the order of convergence, the rate of convergence, and the number of iterations required to obtain sufficiently small errors. 
Let $\{\bm x^{(k)}\}_{k=0}^\infty$ be the sequence of updates generated by a certain iterative first-order method (e.g., PGD). In order to prove the convergence of the algorithm, it is common \cite{luenberger1984linear,polyak1987introduction,bertsekas1997nonlinear,nesterov2003introductory,boyd2004convex} to consider the convergence of the following quantities to $\bm 0$ as $k \to \infty$: {\em (i)} the norm of the generalized gradient ($\unorm{\frac{1}{\eta} (\bm x^{(k+1)} - \bm x^{(k)})}$), {\em (ii)} the gap between current objective function and the optimal value ($\uabs{f(\bm x^{(k)})-f^*}$), and {\em (iii)} the distance to a convergent point ($\unorm{\bm x^{(k)}-\bm x^*}$). Here, we note that $f^*$ and $\bm x^*$ are the limiting points of the objective function $f(\bm x^{(k)})$ and the parameter $\bm x^{(k)}$ as the number of iterations $k$ goes to infinity, respectively.
In (i), the convergence of the generalized gradient norm to $0$ implies the stationarity condition of the constrained problem is satisfied. It follows that the algorithm converges to \textit{a set of stationary points} of the problem.
In (ii), the convergence on the function side is often obtained via the monotonicity of the objective-value sequence $\{ f(\bm x^{(k)}) \}_{k=0}^\infty$ (e.g., decreasing to a limiting value $f^*$). This in turn implies the sequence $\{ \bm x^{(k)} \}_{k=0}^\infty$ converges to \textit{a set of local optima} that yields the same objective function value $f^*$.\footnote{An example for such scenario is minimizing a convex but not strongly convex function $f(\bm x) = \norm{\bm x}_1$ subject to $\bm x \in \R^n$ and $\norm{\bm x}_2^2=1$. The $2n$ vectors $\{\bm e_i\}_{i=1}^n$ and $\{-\bm e_i\}_{i=1}^n$ are local minimizers that obtain the same objective function value. It is worthwhile mentioning that they are also the global solutions of the foregoing problem.}
In (iii), the convergence of $\unorm{\bm x^{(k)}-\bm x^*}$ implies convergence to a unique point that is often an \textit{isolated local optimum point} of the problem. Typically, convergence on the domain side is used in linear convergence proofs for strongly convex settings.

\subsection{Convergence to a Global Optimum} 
In general, a stationary point can be a saddle point, a local/global minimum, or local/global maximum of the problem.
When both the objective function and the constraint set are convex, it is well-known that all stationary points are also global optima of the problem.
Convergence analysis of iterative algorithm (e.g., PGD) in convex optimization therefore focus on providing a universal upper bound on the distance to the global solutions. Analysis on the domain side (iii) is usually used in the presence of \textit{strong convexity} that guarantees the \textit{uniqueness} of the global optimum \cite{luenberger1984linear}-Section~8.6. Without the strong convexity, one may resort to analysis on the function side (ii) in order to prove convergence to \textit{a set} of global optima \cite{beck2017first}-Section~10.4.3.
When convexity is not guaranteed, due to a non-convex objective and/or a non-convex constraint set, convergence analysis has recourse to a set of stationary points by bounding the generalized gradient norm through iterations (i) \cite{bertsekas1997nonlinear}-Section~2.3.2.
Notwithstanding, recent advances in structured non-convex optimization have shed light on convergence guarantees to global solutions of the problem.
By exploiting the special structure of some classes of non-convex problems and using appropriate initialization, PGD can be shown to converge to a unique global optimum despite the non-convexity of these problems.
Examples of such powerful results include sparse recovery with restricted isometry properties \cite{blumensath2009iterative}, matrix completion with incoherence properties \cite{ma2020implicit}, empirical risk minimization with restricted strong convexity and smoothness properties \cite{khanna2018iht}, and spherically constrained quadratic minimization with hidden convexity \cite{beck2018globally}.

\subsection{Convergence to a Local Optimum}
In general non-convex settings, domain-side convergence analysis is restricted to the local region around the convergence point $\bm x^*$. Such points can be a saddle point, a local minimum, or a local maximum of the problem.
The ROC associated with $\bm x^*$ is the neighborhood in which the algorithm (e.g., PGD) is guaranteed to converge to $\bm x^*$ when initialized inside this region.
To a certain extent, the ROC in the aforementioned global convergence analysis is the entire feasible space. However, while global convergence analysis does not require the initialization to be close to the global solution, it often ignores the local structure near the solution needed for establishing sharp bounds on the speed of convergence.
In particular, bounding techniques employed in global convergence analysis hold universally, including worst-case scenarios. Thus, in many problem-specific settings where the solution lies in a benign neighborhood, the global analysis could lead to conservative convergence rate bounds.
As an illustration, in minimizing a smooth and strongly convex function $f$, gradient descent with a fixed step size achieves the rate of convergence at most $(\kappa-1)/(\kappa+1)$ \cite{polyak1963gradient}, where $\kappa$ is the (global) condition number of $f$. Recall that the condition number of a differentiable convex function is the ratio
of its smoothness $L$ to strong convexity $\mu$ \cite{nesterov2003introductory}.
For any quadratic function, this global bound is also an exact and attainable estimate thanks to the fact that the objective curvature is unchanged everywhere. 
For non-quadratic objectives, on the other hand, this global bound may be loose as $\kappa$ takes into account the worst-case scenario, in which the objective function is most ill-conditioned. 
The asymptotic behavior of gradient descent near the solution indeed relies on the condition number of the local Hessian $\kappa(\bm x^*)$ of the objective function, defining as $\lambda_{\max}(\nabla^2 f(\bm x^*)) / \lambda_{\min}(\nabla^2 f(\bm x^*))$. Generally, we have $\mu \leq \lambda_{\min}(\nabla^2 f(\bm x)) \leq \lambda_{\max}(\nabla^2 f(\bm x^*)) \leq L$, for any $\bm x$ in the domain of $f$, which implies $\kappa(\bm x^*) \leq \kappa$.
This local condition number $\kappa(\bm x^*)$ can be significantly smaller than the global condition number $\kappa$ and hence, \textit{a local convergence analysis can yield a tighter bound that reflects the actual convergence speed of the algorithm near the solution}. 
Similar situation also occurs for constrained least squares in which the Hessian restricted to the constrained set can depend on the local structure of the set.

\subsection{Speed of Convergence}
To illustrate the concept of convergence speed, let us consider the convergence on the domain side, i.e., the distance $\norm{\bm x^{(k)} - \bm x^*}$.
Let $\mu$ be a number between $0$ and $1$.
The convergence of $\{\bm x^{(k)}\}_{k=0}^\infty$ to $\bm x^*$ is said to be at rate $\mu \triangleq \mu(\{\bm x^{(k)}\}_{k=0}^\infty)$ if $\mu = \inf_{\{\epsilon_k\}_{k=0}^\infty} \lim_{k \to \infty} {\epsilon_{k+1}}/{\epsilon_k}$, for any monotonically decreasing sequence $\{\epsilon_k\}_{k=0}^\infty$ satisfying $\norm{\bm x^{(k)} - \bm x^*} \leq \epsilon_k$ for all index $k$.
The asymptotic rate of convergence of gradient descent to $\bm x^*$, denoted by $\rho$, is defined by the worst-case rate of convergence among all possible sequences $\{\bm x^{(k)}\}_{k=0}^\infty$ that are generated by the algorithm and converge to $\bm x^*$, i.e., $\rho = \sup_{\{\bm x^{(k)}\}_{k=0}^\infty} \mu (\{\bm x^{(k)}\}_{k=0}^\infty)$.
Depending on the value of $\rho$ in the interval $[0,1]$, the convergence is said to be \textit{sublinear} when $\rho=1$, \textit{linear} when $0<\rho<1$, or \textit{superlinear} when $\rho=0$. The lower the value of $\rho$ is, the faster the speed of convergence is and the fewer the number of iterations needed is to obtain a close approximation of the solution.
Thus, analytical estimation of the convergence rate plays a pivotal role in convergence analysis.
We would like to note two distinct methods for linear convergence rate analysis dating back to the 1960s. 
The first approach was proposed by Polyak \cite{polyak1987introduction}, based on his earlier study into nonlinear difference equations \cite{polyak1964some}. 
The author analyzed the asymptotic convergence of gradient descent for minimizing some objective function $f$.
Assuming $\bm x^*$ is a non-singular local minimum of $f$, Polyak showed that for any $\delta>0$, there exists $\epsilon>0$ such that if $\norm{\bm x^{(0)} - \bm x^*} < \epsilon$ then the sequence $\{\bm x^{(k)} \}_{k=0}^\infty$ generated by gradient descent satisfies
\begin{align*} 
    &\norm{\bm x^{(k)} - \bm x^*} \leq \norm{\bm x^{(0)} - \bm x^*} (\rho+\delta)^k , \numberthis \label{equ:polyak87}
\end{align*}
where $\rho = \max \{ \abs{1-\eta \lambda_{\max}}, \abs{1-\eta \lambda_{\min}} \}$ and $\lambda_{\max}$ and $\lambda_{\min}$ are the largest and smallest eigenvalues of $\nabla^2 f(\bm x^*)$, respectively. Here we emphasize that $f$ does not need to be smooth and strongly convex everywhere but only so around $\bm x^*$. By setting $\eta_{opt} = 2/(\lambda_{\max}+\lambda_{\min})$, the optimal rate of convergence is given by $\rho_{opt} = (\kappa^* - 1)/(\kappa^* + 1)$, where $\kappa^* = \lambda_{\max}/\lambda_{\min}$ is the condition number of the local Hessian $\nabla^2 f(\bm x^*)$. 
When $f$ is a strongly convex quadratic, the local result coincides with the aforementioned global result in \cite{polyak1963gradient} ($\kappa^* = \kappa$). The expression of $\rho$ in (\ref{equ:polyak87}) is called the \textbf{asymptotic convergence rate} of gradient descent with fixed step size $\eta$.\footnote{It is worthwhile to mention that using a similar technique, Nesterov \cite{nesterov2003introductory} proved that the asymptotic rate is at most $\hat{\rho} = (\kappa^* + 1)/(\kappa^* + 3)$. While this bound also exploits the local information of the optimization problem, we note that it is not as tight as the bound in (\ref{equ:polyak87}).}
The second approach was developed by Daniel \cite{daniel1967conjugate} in 1967, while studying gradient descent with \textit{exact line search}, i.e., choosing $\eta$ that minimizes the objective at each iteration.
Utilizing the Kantorovich inequality \cite{kantorovich1948functional}, the author proved that if $\bm x^{(0)}$ is sufficiently close to $\bm x^*$, there exist a constant $\epsilon$ and a sequence $\{q_k\}_{k=0}^\infty$ such that
\begin{align*}
    \norm{\bm x^{(k)} - \bm x^*} \leq \epsilon \prod_{i=0}^k q_i , \qquad \lim_{k \to \infty} q_k = (\kappa^* - 1)/(\kappa^* + 1) .
\end{align*}
Note that here the characteristics of convergence are also exploited through the Hessian $\nabla^2 f(\bm x^*)$. This result was then extended to study the asymptotic convergence of projected gradient descent for constrained optimization \cite{luenberger1972gradient,lichnewsky1979minimisation,gabay1982minimizing}.

\section{Proof of Example~1}
\label{appdx:c2_Ps}

Our goal in this proof is to establish the Lipschitz differentiability of the projection operator onto the unit sphere $\C = \{ \bm x \in \R^n \mid \unorm{\bm x}=1 \}$.
We start by establishing the Lipschitz differentiability at a point on $\C$ and then extend it to any nonzero point in $\R^n$. For the Lipschitz differentiability on $\C$, we introduce the following lemma:
\begin{lemma} \label{lem:sphere_unit}
For any $\bm x^* \in \C$, we have
\begin{align} \label{equ:PS_xS}
    \sup_{\bm y \in \Pi_{\C}(\bm x^* + \bm \delta)} \unorm{\bm y - \bm x^* - \bigl( \bm I - \bm x^* (\bm x^*)^\topnew \bigr) \bm \delta} \leq 2 \unorm{\bm \delta}^2.
\end{align}
\end{lemma}

\begin{proof}
We consider two cases:

\noindent \textbf{Case 1:} If $\bm x^* + \bm \delta = \bm 0$, then $\Pi_\C(\bm 0) = \C$ and $\unorm{\bm \delta} = \unorm{\bm x^*} = 1$. For any $\bm y \in \C$, substituting $\bm \delta = - \bm x^*$ and then using the fact that $\bm I - \bm x^* (\bm x^*)^\topnew$ is the projection onto the null space of $\bm x^*$, we have
\begin{align*}
    \bm y - \bm x^* - \bigl( \bm I - \bm x^* (\bm x^*)^\topnew \bigr) \bm \delta &= \bm y - \bm x^* + \bigl( \bm I - \bm x^* (\bm x^*)^\topnew \bigr) \bm x^* \\
    &= \bm y - \bm x^* .
\end{align*}
Next, taking the norm and using the triangle inequality yield
\begin{align*}
    \unorm{\bm y - \bm x^* - \bigl( \bm I - \bm x^* (\bm x^*)^\topnew \bigr) \bm \delta} &= \unorm{\bm y - \bm x^*} \\
    &\leq \unorm{\bm y} + \unorm{\bm x^*} = 2 \unorm{\bm \delta}^2 ,
\end{align*}
where the last step stems from $\norm{\bm y} = \norm{\bm x^*} = \norm{\bm \delta} = 1$.
Thus, (\ref{equ:PS_xS}) holds in this case.
    
\noindent \textbf{Case 2:} If $\bm x^* + \bm \delta \neq \bm 0$, then $\Pi_\C(\bm x^* + \bm \delta)$ is singleton containing the unique projection
\begin{align*}
    \P_\C(\bm x^* + \bm \delta) = \frac{\bm x^* + \bm \delta}{\unorm{\bm x^* + \bm \delta}} .
\end{align*}    
Hence, (\ref{equ:PS_xS}) is equivalent to 
\begin{align} \label{equ:PS_xS_1}
    \norm{\frac{\bm x^* + \bm \delta}{\unorm{\bm x^* + \bm \delta}} - \bm x^* - \bigl( \bm I - \bm x^* (\bm x^*)^\topnew \bigr) \bm \delta} \leq 2 \unorm{\bm \delta}^2 .
\end{align}
We prove (\ref{equ:PS_xS_1}) by {\em(i)} showing that for any scalars $u>0$ and $(1-u)^2 \leq v \leq (1+u)^2$:
\begin{align} \label{equ:uv_sphere}
    (17u-2)v^2 - 2u(1-u)^2v + (1-u)^4(u+2) \geq 0 ,
\end{align}
and {\em(ii)} showing that (\ref{equ:uv_sphere}) is equivalent to (\ref{equ:PS_xS_1}) with $u=\unorm{\bm x^* + \bm \delta} > 0$ and $v=\unorm{\bm \delta}^2 \geq 0$.

\noindent {\em(i)} To prove (\ref{equ:uv_sphere}), let us consider the following cases:
\begin{enumerate}
    \item If $0<u \leq {2}/{17}$, then for $v\leq (1+u)^2$, we have
    \begin{align*}
        (17u-2)v^2 &- 2u(1-u)^2v + (1-u)^4(u+2) \\
        &\geq (17u-2)(1+u)^4 - 2u(1-u)^2 (1+u)^2 \\
        &\qquad + (1-u)^4(u+2) \\
        &= 16u^2 (u+2)(u^2+2u+2) \geq 0 .
    \end{align*}
    \item If ${2}/{17}< u \leq {1}/{2}$, then for $(1-u)^2 \leq v \leq (1+u)^2$, the following holds
    \begin{align*}
        (17u-2)v^2 &- 2u(1-u)^2v + (1-u)^4(u+2) \\
        &\geq (17u-2)(1-u)^4 - 2u(1-u)^2 (1+u)^2 \\
        &\qquad + (1-u)^4(u+2) \\
        &= 8u(1-u)^2 (2-u)(1-2u) \geq 0 .
    \end{align*}
    \item If $u>{1}/{2}$, using the quadratic vertex at $v=u(1-u)^2/(17u-2)$ as the minimum point, we obtain
    \begin{align*}
        (17u-2)v^2 &- 2u(1-u)^2v + (1-u)^4(u+2) \\
        &\geq \frac{4(1-u)^4(4u^2+8u-1)}{17u-2} \geq 0 .
    \end{align*}
\end{enumerate}

\noindent {\em(ii)} Now for $u=\unorm{\bm x^* + \bm \delta} > 0$ and $v=\unorm{\bm \delta}^2 \geq 0$, we have $(\bm x^*)^\topnew \bm \delta = (u^2-v-1)/2$ and
\begin{align*}
    (\ref{equ:PS_xS_1}) \Leftrightarrow ~ &\norm{\frac{\bm x^* + \bm \delta}{\unorm{\bm x^* + \bm \delta}} - \bm x^* - \bigl( \bm I - \bm x^* (\bm x^*)^\topnew \bigr) \bm \delta} \leq 2 \unorm{\bm \delta}^2 \\
    \Leftrightarrow ~ &\unorm{\bm x^* + \bm \delta - \unorm{\bm x^* + \bm \delta} (\bm x^* + \bm \delta - \bm x^* (\bm x^*)^\topnew \bm \delta)}^2 \\
    &\qquad \qquad \leq 4 \unorm{\bm x^* + \bm \delta}^2 \unorm{\bm \delta}^4 \\
    \Leftrightarrow ~ &\unorm{(1 - u)(\bm x^* + \bm \delta) + u ((\bm x^*)^\topnew \bm \delta) \bm x^* }_2^2 \leq 4 u^2 v^2 \\
    \Leftrightarrow ~ &(1-u)^2u^2 + u^2 \Bigl( \frac{u^2-v-1}{2} \Bigr)^2 \\
    &+ 2u(1-u)\frac{u^2-v-1}{2}\frac{u^2-v+1}{2} \leq 4 u^2 v^2  \\
    \Leftrightarrow ~ &(\ref{equ:uv_sphere}) .
\end{align*} 
Finally, by the triangle inequality, we have
\begin{align*}
    \uabs{\unorm{\bm x^* + \bm \delta} - \unorm{\bm x^*}} \leq \unorm{\bm \delta} \leq \unorm{-\bm x^*} + \unorm{\bm x^* + \bm \delta} ,
\end{align*}
which in turn verifies $(1-u)^2 \leq v \leq (1+u)^2$.
This completes our proof of the lemma.
\end{proof}

Next, to extend the result in Lemma~\ref{lem:sphere_unit} to any $\bm x \in \R \setminus \{0\}$, we substitute $\bm x^* = \bm x/\unorm{\bm x}$ and $\bm \delta = \bm \delta/\unorm{\bm x}$ into (\ref{equ:PS_xS}) and obtain
\begin{align*}
    \sup_{\bm y \in \Pi_{\C}(\frac{\bm x}{\unorm{\bm x}} + \frac{\bm \delta}{\unorm{\bm x}})} \norm{\bm y - \frac{\bm x}{\unorm{\bm x}} - \Bigl( \bm I_n - \frac{\bm x \bm x^{\topnew}}{\unorm{\bm x}^2} \Bigr) \frac{\bm \delta}{\unorm{\bm x}}} \leq 2 \frac{\unorm{\bm \delta}^2}{\unorm{\bm x}^2} . \numberthis \label{equ:S1}
\end{align*}
Since the projection onto the unit sphere is scale-invariant,
\begin{align*}
    \Pi_{\C}\Bigl(\frac{\bm x}{\unorm{\bm x}} + \frac{\bm \delta}{\unorm{\bm x}} \Bigr) = \Pi_\C (\bm x + \bm \delta) . \numberthis \label{equ:S2}
\end{align*}
Substituting (\ref{equ:S2}) into (\ref{equ:S1}) yields (6). 
Thus, by Definition~2, for any ${\bm x} \neq {\bm 0}$ we obtain
\begin{align*}
    &\nabla \P_\C(\bm x) = \frac{1}{\unorm{\bm x}} \Bigl( \bm I_n - \frac{\bm x \bm x^{\topnew}}{\unorm{\bm x}^2} \Bigr) , \\
    &c_1(\bm x)=\infty , \quad c_2(\bm x)=\frac{2}{\unorm{\bm x}^2} .
\end{align*}

\section{Details of Application B - Iterative Hard Thresholding for Sparse Recovery}
\label{appdx:IHT_counter}

\subsection{Proof of (42)}

In this subsection, we first show that any $\bm x^* \in \Phi_{= s}$ and $\bm x \in \B(\bm x^*, \uabs{x^*_{[s]}}/\sqrt{2})$ share the same index set of $s$-largest elements (in magnitude), i.e., $\Omega_s(\bm x^*)$. Then, we construct a counter-example to demonstrate that $\uabs{x^*_{[s]}}/\sqrt{2}$ is the largest possible radius so that (42) holds.

First, we show that for any $i \in \Omega_s(\bm x^*)$ and $j \in \{1,\ldots,n\} \setminus \Omega_s(\bm x^*)$, $\uabs{x_j} < \uabs{x_i}$ as follows.
In particular, we have
\begin{align*}
    \uabs{x_{j} - x^*_{j}} + \uabs{x_{i} - x^*_{i}} &\leq \sqrt{2 ( (x_{j} - x^*_{[j]})^2 + (x_{i} - x^*_{i})^2 )} \\
    &\leq \sqrt{2 \unorm{\bm x - \bm x^*}^2} < \uabs{x^*_{[s]}} ,
\end{align*}
where the last inequality stems from the fact that $\unorm{\bm x - \bm x^*} < \uabs{x^*_{[s]}}/\sqrt{2}$. 
Now, since $x^*_j = 0$ for all $j \in \{1,\ldots,n\} \setminus \Omega_s(\bm x^*)$, we have
\begin{align*}
    \uabs{x_{j}} &= \uabs{x_{j} - x^*_{j}} \\
    &< \uabs{x^*_{[s]}} - \uabs{x_{i} - x^*_{i}} \\
    &\leq \uabs{x^*_{i}} - \uabs{x_{i} - x^*_{i}} \\
    &\leq \uabs{x^*_{i} + (x_{i} - x^*_{i})} = \uabs{x_{i}} , \numberthis \label{equ:arg_sparse}
\end{align*}
Therefore, every $\bm x \in \B(\bm x^*,\uabs{x^*_{[s]}}/\sqrt{2})$ shares the same index set of $s$-largest (in magnitude) elements with $\bm x^*$, i.e., $\Omega_s(\bm x) = \Omega_s(\bm x^*)$, which implies (42).

We now construct the counter-example as a point $\bm x$ such that $\Omega_s(\bm x) \neq \Omega_s(\bm x^*)$ and $\bm x$ is not in $\B(\bm x^*,\uabs{x^*_{[s]}}/\sqrt{2})$ but arbitrarily close to its boundary.
Without loss of generality, assume that $\uabs{x^*_1} \geq \ldots \geq \uabs{x^*_s} > \uabs{x^*_{s+1}} = \ldots = \uabs{x^*_n} = 0$.
For arbitrarily small $\epsilon>0$, define $\bm x$ as
\begin{align*}
    x_i = \begin{cases}
        x^*_s/2 &\text{if } i=s , \\
        x^*_s/2 + \epsilon &\text{if } i=s+1 , \\
        x_i &\text{otherwise.}
    \end{cases}
\end{align*}
Then, since $x_{s+1}<x_s$, $\bm x$ does not shares the same index set of $s$-largest (in magnitude) elements with $\bm x^*$. On the other hand, as $\epsilon \to 0$, we have
\begin{align*}
    \unorm{\bm x - \bm x^*} &= \sqrt{\sum_{i=1}^n (x_i - x^*_i)^2} \\ 
    &= \sqrt{\Bigl(-\frac{x^*_s}{2}\Bigr)^2 + \Bigl( \frac{x^*_s}{2}+\epsilon \Bigr)^2} \to \frac{1}{\sqrt{2}} \uabs{x^*_{[s]}} .
\end{align*}
This means $\bm x \not \in \B(\bm x^*,\uabs{x^*_{[s]}}/\sqrt{2})$ but it can approach the boundary of the ball as $\epsilon$ decreases to $0$.

\subsection{Proof of Remark~6}

In the following, we show any stationary point $\bm x^*$ of (40) is also a local minimum by proving that the objective function does not decrease if we add any perturbation to $\bm x^*$ on $\C$. 
Let us consider any perturbation $\bm \delta$ such that $\bm \delta \in \B(\bm 0, c_1(\bm x^*))$ and $\bm x = \bm x^* + \bm \delta \in \C$. Since $\bm x \in \B(\bm x^*,c_1(\bm x^*))$, using (\ref{equ:arg_sparse}), we have $\uabs{x_{[1]}} \geq \ldots \uabs{x_{[s]}} > 0$.
On the other hand, since $\bm x$ has no more than $s$ non-zero entries, it must hold that $\uabs{x_{[s+1]}} = \ldots = \uabs{x_{[n]}} = 0$. 
Therefore, $\bm x = \bm S_{\bm x^*} \bm S_{\bm x^*}^{\topnew} \bm x$, which implies $\bm \delta = \bm S_{\bm x^*} \bm S_{\bm x^*}^{\topnew} \bm \delta$.
Now we represent the change in the objective function as
\begin{align*}
    &\frac{1}{2} \unorm{\bm A (\bm x^* + \bm \delta) - \bm b}^2 - \frac{1}{2} \unorm{\bm A \bm x^* - \bm b}^2 \\
    &= \frac{1}{2} \bm \delta^\topnew \bm A^\topnew \bm A \bm \delta + \bm \delta^\topnew \bm A^\topnew ( \bm A \bm x^* - \bm b ) \\
    &= \frac{1}{2} \bm \delta^\topnew \bm S_{\bm x^*} \bm S_{\bm x^*}^{\topnew} \bm A^\topnew \bm A \bm S_{\bm x^*} \bm S_{\bm x^*}^{\topnew} \bm \delta + \bm \delta^\topnew \bm S_{\bm x^*} \bm S_{\bm x^*}^{\topnew} \bm A^\topnew ( \bm A \bm x^* - \bm b ) \\
    &= \frac{1}{2} \bm \delta^\topnew \bm S_{\bm x^*} (\bm S_{\bm x^*}^{\topnew} \bm A^\topnew \bm A \bm S_{\bm x^*}) \bm S_{\bm x^*}^{\topnew} \bm \delta \geq 0 , \numberthis \label{equ:sparse_f2}
\end{align*}
where the last equality uses the stationarity condition in (43).
From (\ref{equ:sparse_f2}), we conclude $\bm x^*$ is a local minimum of (40).

\end{document}